\documentclass[11pt,reqno]{amsproc}

\usepackage[margin=1.0in]{geometry}
\usepackage[english]{babel}
\usepackage[utf8]{inputenc}
\usepackage{ulem}
\usepackage{amsmath,amssymb,amsthm, comment, mathtools}
\usepackage{hyperref,mathrsfs,xcolor}
\usepackage{enumitem}

\usepackage{tikz}
\usetikzlibrary{patterns}
\usetikzlibrary{fpu}
\usetikzlibrary{plotmarks}
\usetikzlibrary{decorations.markings}
\usepackage{pgfplots}
\usetikzlibrary{intersections, pgfplots.fillbetween}

\usepackage{genealogytree}

\mathtoolsset{showonlyrefs}

\newtheorem{proposition}{Proposition}[section]

\newcommand{\R}{\mathbb{R}}

\newcommand{\N}{\mathcal{N}}
\newcommand{\x}{\mathbf{x}}

\newcommand{\Z}{\mathbb{Z}}

\newcommand{\B}{\mathcal{B}}

\newcommand{\pa}{\partial}
\newcommand{\ve}{\varepsilon}
\newcommand{\rmd}{{\rm d}}

\usepackage[mathscr]{euscript}

\newcommand{\eee}{equation}
\newcommand{\be}{\begin{\eee}}
\newcommand{\ee}{\end{\eee}}

\numberwithin{equation}{section}
\newtheorem{theorem}{Theorem}
\numberwithin{theorem}{section}

\newtheorem{question}{Question}
\numberwithin{prop}{section}

\numberwithin{cor}{section}
\newtheorem{lemma}{Lemma}
\numberwithin{lemma}{section}
\theoremstyle{remark}
\newtheorem*{remark}{Remark}
\theoremstyle{definition}

\title{\vspace*{-25mm} On the fragility of laminar flow}

\author{Theodore D. Drivas}
\address{Department of Mathematics, Stony Brook University, Stony Brook, NY, 11790}
\email{tdrivas@math.stonybrook.edu}

\author{Daniel Ginsberg}
\address{Department of Mathematics, Brooklyn College (CUNY), Brooklyn, NY 11210, USA}
\email{daniel.ginsberg@brooklyn.cuny.edu}

\author{Marc Nualart}
\address{Instituto de Ciencias Matem\'{a}ticas, Consejo Superior de Investigaciones Cient\'{i}ficas, 28049 Madrid, Spain}
\email{marc.nualart@icmat.es}

\begin{document}
\maketitle
\vspace{-5mm}
\begin{abstract}
Inviscid laminar flow is a stationary solution of the incompressible Euler equations whose streamlines foliate the fluid domain.  Their structure on symmetric domains is rigid:
all laminar flows occupying straight periodic channels are shear and on regular annuli they are circular \cite{drivas2024geometric}.  Laminarity can persist to slight deformations of these domains provided the base flow is Arnold stable and \textit{non-stagnant} (non-vanishing velocity) \cite{CDG21}. On the other hand, flows with \textit{trivial net momentum} (and thus stagnate) break laminarity by developing islands (regions of contractible streamlines) on all non-flat periodic channels with up/down reflection symmetry \cite{drivas2024islands}.
Here, we show that stable steady states occupying \textit{generic} channels or annuli and stagnate must have islands. Additionally, when the domain is close to symmetric, we characterize the size of the islands, showing that they scale as the square root of the boundary's deviation from flat.  Taken together, these results show that dynamically stable laminar flows are structurally unstable whenever they stagnate.
\end{abstract}

\section{Introduction}
In this paper, we are interested in the existence of islands (regions of contractible streamlines) for steady fluid flow  confined in periodic domains
\begin{equation}
D_{G,H} := \lbrace (x,y)  : \, x\in \mathbb{T}, \, G(x) \leq y \leq H(x) \rbrace.
\end{equation}
Here, $G,H\in C^{2,\alpha}(\mathbb{T},\R)$ for $\alpha\in(0,1)$ are any two boundary defining functions, with $G(x) < H(x)$, for all $x\in \mathbb{T}$.
The steady incompressible Euler equations (perfect fluids) read
\begin{equation}\label{eq:euler}
\begin{cases}
&\nabla^\perp \psi\cdot \nabla \omega=0 \qquad \text{in} \qquad  D_{G,H} \\
&\Delta \psi =\omega \qquad \qquad \ \ \text{in} \qquad  D_{G,H}\\
&\partial_\tau \psi  = 0 \qquad \qquad\ \ \text{on} \qquad  \partial D_{G,H}
\end{cases}
\end{equation}
and $\partial_\tau$ denotes the derivative tangent to $\partial D_{G, H}$.
There is a one parameter freedom in the above formulation which is fixed by prescribing $\int_{\partial  D_{G,H}^{\mathsf top}} \partial_n \psi \rmd \ell = \gamma$ where  $\gamma \in \mathbb{R}$ is a prescribed circulation, which is preserved on dynamical solutions in view of the Kelvin theorem. The circulation condition determines the cohomological character of the solution, e.g. the projection of $\psi$ onto harmonic functions on the  domain $D_{G,H}$.    See \cite{drivas2023singularity}. A special but important class of solutions to  \eqref{eq:euler} are those with a specified functional relationship between vorticity and streamfunction.  Specifically, for any $F\in C^1$, and $c_G,c_H\in \R$, solutions $\psi$ to  the following semi-linear elliptic problem
\begin{equation}\label{eq:steadyEulerDve}
\begin{cases}
\Delta\psi = F(\psi) & \text{ in } D_{G,H}, \\
\psi(x,H(x))=c_H & \text{ for } x\in\mathbb{T}, \\
\psi(x,G(x))=c_G & \text{ for } x\in\mathbb{T},
\end{cases}
\end{equation}
give rise to stationary configurations of the Euler equations in $D_{G,H}$, for some circulation $\gamma$. 
We are interested in the structure of these steady states.  In some cases, their streamlines $\{\psi= {\rm const}\}$ laminate the domain (see \cite{CDG21} and also \cite{hamel2017shear, hamel2019liouville, HM23,drivas2024geometric}).  In others, islands of contractible streamlines must exist (see \cite{drivas2024islands} and also \cite{coti2023stationary, nualart2023zonal}).
The purpose of this paper is to show that stable stationary solutions $\psi$ of the Euler equations confined to generic domains $D_{G,H}$ possess islands. 

\begin{theorem}\label{thm:genericityisland}
Fix $\alpha\in (0,1)$ and $F\in C^4(\mathbb{R};\mathbb{R})$ satisfying  $F'>-\lambda_1$, where $\lambda_1$ is the first eigenvalue of the Dirichlet Laplacian in $D_{G,H}$. There exists an open dense set $\B\subset (C^{2,\alpha}(\mathbb{T}))^2$ so that for each $(G,H)\in \B$,  any solution $\psi$ to \eqref{eq:steadyEulerDve}   on $D_{G,H}$  having stagnation points possess islands. 
\end{theorem}
Thus, either the flow is laminar because the flow is non-stagnant \cite{CDG21}, or -- if stagnation sets exist -- then generically they must be centers of islands separated by hyperbolic stagnation points. We remark that the complement of $\mathcal{B}$ is non-empty. Indeed, for any analytic $F$ and analytic non-flat non-contractible curve $\Gamma_0$ in $\mathbb{T}\times\R$, one can produce, by means of Cauchy-Kovalevskaya, laminar solutions $\psi$ to \eqref{eq:steadyEulerDve} in $D_{\Gamma_G,\Gamma_H}$ such that $\nabla\psi=0$ on $\Gamma_0$ and $\Gamma_G$ and $\Gamma_H$ are non-flat level sets of $\psi_0$. Such a domain may be very thin but nevertheless provides examples of channels which support laminar flows with essentially any geometry. See also \cite[Theorem 4.12]{elgindi2024classification} for such a construction for special $F$ of finite regularity.

An important class of steady states for which Theorem \ref{thm:genericityisland} applies are Arnold stable steady states (see \cite{arnold2009topological, arnold1966geometrie}) and highlights the drastic consequences of small boundary perturbations in stable fluid equilibria. Indeed, while any Arnold stable steady state in $\mathbb{T}\times[-1,1]$ is a shear flow with closed non-contractible fluid trajectories, we show that arbitrarily small disturbances of the boundary induce a break in the streamline topology of Arnold stable equilibria with the appearance of islands.  Note  also that the existence of a stagnation point is guaranteed by Poincare's last geometric theorem as soon as the boundary velocities point in different directions.

The ideas presented in this paper allow us also to conclude the appearance of islands for steady Euler solutions  in generic annular domains posed in $\mathbb{T}\times \R$ or also in $\mathbb{R}^2$, no longer assuming that the boundaries are graphs in the periodic variable. Namely,  steady Euler solutions that stagnate on generic annular domains  must have islands. This extension is described briefly  at the end of \S \ref{sec:islandgenericdomain}.

In some perturbative settings, we can understand more about the structure of the set $\B$.  For instance, we consider small perturbations of the straight channel 
\be
G=-1 + \ve g, \quad H=1+ \ve h.
\ee
If furthermore $c_H=c_G$, then the set $\B$ locally contains
\begin{equation}\label{eq:defB}
\B'=\lbrace (-1+ \ve g,1+ \ve h) \ : (g, h)\in (C^{2,\alpha}(\mathbb{T}))^2, \quad  h'(x)+g'(x)\not\equiv 0 \rbrace,
\end{equation}
which is clearly open and dense in $(C^{2,\alpha}(\mathbb{T}))^2$ nearby $(G,H)=(-1,1)$.

Once the existence of islands is ensured, a natural problem is to quantify their size in terms of the size of the boundary perturbation, which is of order $\ve$. Since islands appear nearby local critical points of $\psi_\ve$, they can be found to be arbitrarily small. Instead, we are interested in describing the maximal size they can reach in terms of $\ve$. This is the purpose of our next result, which describes the maximal size of islands appearing in perturbations of the periodic channel $D_{-1,1}=\mathbb{T}\times(-1,1)$.
\begin{theorem}\label{thm:genericsizeisland}
Under the assumptions of Theorem \ref{thm:genericityisland}, there exists an open and dense set $\B_0\subset \B'$ so that for all $(g,h)\in \B_0$ there exists $\ve_*>0$ such that the solution $\psi_\ve$  to \eqref{eq:steadyEulerDve} in $D_{-1+\ve g,1+\ve h}$ possess islands whose maximal height is at least of order $\sqrt{\ve}$, for all $0 < \ve < \ve_*$.
\end{theorem}
The set $\B_0$ contains boundary perturbations for which solutions to some elliptic equation have non-degenerate local maxima, and its precise description is postponed until Section \ref{sec:sizeislands}.

    \begin{figure}[h!]\label{bstruct}
      \includegraphics[height=.3 \linewidth]{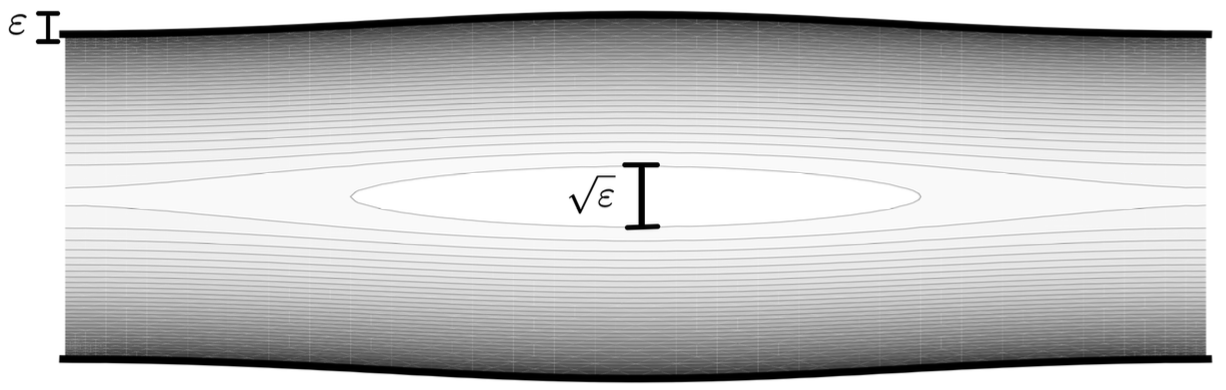}
    \end{figure}

We now explain the heuristic reason for the generic existence of islands. 
Consider any steady Euler solution $u=\nabla^\perp\psi$ in $D_{G,H} $ which has an interior critical point. If $\psi$ does not have islands, then it is laminar and by Proposition \ref{prop:islandscritical}, its set of critical points is connected and belongs to some closed non-contractible level set $\Gamma_0$. In particular, $\nabla\psi=0$ and $\psi=c$ on $\Gamma_0$. If $\omega = F(\psi)$ where $F$ is analytic, then $\psi$ is uniquely determined by $\Gamma_0$ by the Cauchy-Kovalevskaya theorem. See \cite{elgindi2024classification} where such a construction is carried out to obtain laminar steady Euler solutions emanating from a singular streamline. Since the level sets of $\psi$ are assumed to foliate the domain, they must do so monotonically on each side of $\Gamma_0$ and at some point they must meet the top and bottom boundaries. If the boundaries do not coincide with  level sets of $\psi$, then $\psi$ has islands.

Since $\psi$ is uniquely determined by the curve $\Gamma_0$, $c\in \R$ and $F$, one would conjecture that there is a one-to-one correspondence between the singular streamline $\Gamma_0$ and the boundary defining functions $(G,H)$ for which the level sets $\lbrace\psi = c_G \rbrace$ and $\lbrace\psi = c_H \rbrace$ coincides with the domain boundaries, so that the image of the map $\Gamma_0\mapsto (\lbrace\psi = c_G \rbrace,\lbrace\psi = c_H \rbrace)$ is nowhere dense. Indeed, there is one function-worth degree of freedom in $\Gamma_0$, but two defining functions $H$ and $G$ of the boundaries. In this regard,  the image of the map has empty intersection with the set of boundary perturbations where one of them is constant, see Proposition \ref{prop:ctnbottombdry}. 
Likewise, Lemma \ref{lemma:pxvarphi0} shows that for all $(g,h)\in \mathcal{B}'$, the pair $(-1-\ve g,1+\ve h)$ is not in the image of the map $\Gamma_0\mapsto (\lbrace\psi = c_G \rbrace,\lbrace\psi = c_H \rbrace)$, for $\ve>0$ small enough. 

By continuity, the image of the map is closed, and hence it is nowhere dense if it has empty interior. Then, for any analytic $(G,H)$ and $F$ for which the solution $\psi$ to \eqref{eq:steadyEulerDve} does not have islands, it is enough to show the existence of some $(g,h)$ for which the solution $\psi_\ve$ of \eqref{eq:steadyEulerDve} in $D_{G+\ve g, H+\ve h}$ develops islands, for $\ve>0$ sufficiently small. This is precisely the main idea behind Theorem \ref{thm:genericityisland} and consists of finding the solution $\psi_\ve$ as the sum of $\psi$, a first correction of size $\ve$ and a second correction of size $\ve^2$. The first order correction solves a linear equation in $D_{G,H}$ with boundary conditions linearly prescribed by the pair $(g,h)$, see \eqref{eq:HGvarphi}, and must be constant along $\Gamma_0$ if $\psi_\ve$ does not have islands. Thus, finding an instance of $(g,h)$ for which the first order correction is non-constant along $\Gamma_0$ ensures the existence of islands for $\psi_\ve$ and proves that the image of the map is nowhere dense.

The outline of the paper is as follows. Section \ref{sec:islandgenericdomain} provides some general criterion for the existence of islands, and applies them to prove Theorem \ref{thm:genericityisland}.  Section \ref{sec:pert} gives a proof in the case of small perturbations of the flat channel, where more information about the generic set of boundaries can be provided.  Section \ref{sec:sizeislands} establishes a lower bound on the maximal height of the islands, proving Theorem \ref{thm:genericsizeisland}.  Appendix \ref{appendislands1} gives some criteria for existence of islands in some definite cases.  Appendix \ref{estappend} contains some technical estimates for estimating a remainder on an asymptotic expansion used to prove Theorem \ref{thm:genericityisland}.

\section{Islands for generic domains}\label{sec:islandgenericdomain}
In this section we show Theorem \ref{thm:genericityisland}. The first part of the section is devoted to proving the generic appearance of islands in $D_\ve := D_{G+\ve g, H+\ve h}$ if they are absent in $D_0:=D_{G,H}$. Throughout the section, unless stated otherwise, we assume $\psi_0$ is a laminar non-constant solution to \eqref{eq:steadyEulerDve}. Let $\Gamma_0$ be the unique (otherwise Proposition \ref{prop:islandscritical} applies) streamline of $\psi_0$ on which $\nabla\psi_0=0$. Our first result shows that $\partial_y\psi_0$ (and thus the fluid velocity) is not identically zero on the boundaries.

\begin{lemma}\label{lemma:nonzerofluidvel}
If $\psi_0$ is laminar, then $\partial_y\psi_0\not\equiv 0$ on each connected component of $\partial D_{G,H}$.
\end{lemma}

\begin{proof}
Assume $\partial_y\psi_0(x,G(x))=0$ for all $x\in \mathbb{T}$. Since $\psi_0(x,G(x))=c_G$, we deduce that $\nabla\psi(x,G(x))\equiv 0$. Then, $\psi_0$ is laminar and solves
\begin{equation}\label{eq:steadyEulerfreebdyDHG}
\begin{cases}
\Delta\psi_0 = F(\psi_0) & \text{ in } D_{G,H}, \\
\psi_0=c_0, \quad \nabla\psi_0=0  & \text{ on } \Gamma_0, \\
\psi_0(x,G(x)) = c_G, \quad \nabla\psi_0(x,G(x))=0 & \text{ for } x\in\mathbb{T}
\end{cases}
\end{equation}
Hence, from Corollary 1.1 in \cite{drivas2024geometric} we deduce that $\psi_0$ is a shear flow, that $G(x)=G_0$ and $\Gamma_0=\lbrace (x,y_0) : x\in \mathbb{T} \rbrace$ for some $G_0<y_0<H(x)$, for some constant $G_0 < \min H(x)$. By elliptic stability of $F$ and $\nabla\psi_0=0$ on $\partial D_{G_0,\Gamma_0}$, we deduce that $\nabla\psi_0=0$ in $D_{G_0,\Gamma_0}$. Hence, $\psi_0$ is constant in $D_{G_0,\Gamma_0}$. Moreover,  we see that either $H'\equiv 0$ or $\psi_0$ is constant on $D_{G,H}$ by using Lemma 3 in \cite{drivas2024islands}. Since $\psi_0$ is assumed non-constant on $D_{G,H}$, we have that $H=H_0$ for some $H_0>G_0$. Then, $\psi_0(x,y) = \psi_0(y)$ due to $\psi_0(x,H_0)=c_H$ and the elliptic stability of $F$. Therefore, $\psi_0(y)$ satisfies $\psi_0''(y)=F(\psi_0(y))$, with $\psi_0(y_0)=c_0=c_G$ and $\psi_0'(y_0)=\psi_0''(y_0)=0$ due to continuity and $\psi_0=c_G$ in $D_{G,\Gamma_0}$. Thus, $F(c_0)=0$ and by the uniqueness of solutions to the ODE we conclude that $\psi_0=c_G$ is constant, reaching a contradiction.
\end{proof}

Next, we parametrize $\Gamma_0$ by a smooth graph. For this, we first need to show that $\psi_0$ is non-degenerate on $\Gamma_0$.
\begin{lemma}\label{lemma:nonzeroFc0}
Let $c_0=\psi_0|_{\Gamma_0}$. Then, $F(c_0)\neq 0$ and the curve $\Gamma_0$ is a smooth graph in $x$.
\end{lemma}

\begin{proof}
Assume $F(c_0)=0$. Then, since $\nabla\psi_0=0$ along $\Gamma_0$, let $y_* < \min \Gamma_0$ and define $\Gamma_* = \lbrace y = y_* \rbrace$. In $D_{\Gamma_*,H}$ we further extend $\psi_0$ to $\widetilde{\psi}_0$ defined by
\begin{align*}
\widetilde{\psi}_0 = \begin{cases}
\psi_0 & \text{in }D_{\Gamma_0,H}, \\
c_0 & \text{in }D_{\Gamma_*,\Gamma_0},
\end{cases}
\end{align*}
and we observe that $\widetilde{\psi}_0\in H^2(D_{\Gamma_*,H})$ is a solution to
\begin{equation}\label{eq:extendedpsi0}
\begin{cases}
\Delta\widetilde{\psi}_0 = F(\widetilde{\psi}_0) & \text{ in } D_{\Gamma_*,H}, \\
\psi_0=c_0, \quad \nabla\psi_0=0  & \text{ on } \Gamma_*=\lbrace y = y_* \rbrace. \\
\end{cases}
\end{equation}
Hence, \cite[Lemma 3]{drivas2024islands} shows that $\partial_x\widetilde{\psi}_0=0$ in $D_{\Gamma_*,H}$ and either $H(x)=H_0$ or $\psi_0$ is constant in $D_{\Gamma_*,H}$. If $H(x)=H_0$, since $\widetilde{\psi}_0(x,y) = \psi_0(y)$ satisfies \eqref{eq:extendedpsi0} in $D_{\Gamma_0,H_0}$ and $\partial_y\psi_0(y_0)=\partial_y^2\psi_0(y_0)=0$ for some $y_0<H_0$, we likewise conclude that $\psi_0$ is constant in $D_{\Gamma_*,H_0}$. Arguing similarly for the bottom boundary, we conclude that $\psi_0$ is constant in $D_{G,H}$, a contradiction.

Now, since $F(c_0)\neq 0$, we then can parametrize $\Gamma_0$ by a smooth curve, confer \cite[Lemma 2.3]{drivas2024geometric}. Moreover, an inspection of the proof of \cite[Lemma 5.1]{drivas2024geometric} shows that the smooth curve $\Gamma_0$ is given by a graph in $x$, namely $\Gamma_0 = \lbrace (x,y_0(x)): x\in \mathbb{T} \rbrace$ for some $y_0\in C^1(\mathbb{T})$. Indeed, Lemma 5.1 in \cite{drivas2024geometric} relies on a moving plane argument that shows $\partial_y^2\psi_0\neq 0$ on $\Gamma_0$. There, the moving plane process stops whenever the reflection of any connected component of the boundary of $D_{G,\Gamma_0}$ with respect to some plane $\lbrace y = \text{const} \rbrace$ become internally tangent with itself, or the boundary becomes orthogonal to the plane. Since reflected graphs do not intersect with themselves and they are not orthogonal to horizontal planes, the moving plane process, if it stops, must only do so because of the free boundary $\Gamma_0$, since the bottom boundary is given by the graph defined by $G$. Then, one reaches a contradiction with the Hopf Lemma, see \cite[Lemma 5.1]{drivas2024geometric} for the details.

Since $\nabla\psi_0(x,y_0(x))$ for all $x\in \mathbb{T}$ we have that
\begin{align*}
0 &= \partial_x^2\psi_0(x,y_0(x)) + y_0'(x)\partial_{xy}^2\psi_0(x,y_0(x)), \\ 
0&=\partial_{x,y}^2\psi_0(x,y_0(x)) + y_0'(x)\partial_{y}^2\psi_0(x,y_0(x)),
\end{align*}
from which we deduce that $\partial_x^2\psi_0(x,y_0(x)) = (y_0'(x))^2\partial_y^2\psi_0(x,y_0(x))$. In particular, we readily see that $\partial_y^2\psi_0(x,y_0(x))\neq 0$ for all $x\in \mathbb{T}$, since otherwise we would have $\partial_x^2\psi_0(x,y_0(x))=0$ as well and then $F(\psi_0(x,y_0(x)))=F(c_0)=0$, a contradiction.
\end{proof}

\subsection{Asymptotic expansion of the streamfunction }In this subsection we further investigate the solution $\psi_\ve$ to \eqref{eq:steadyEulerDve} posed in $D_{G+\ve g,H+\ve h}$. The following result shows that $\psi_\ve$ is a linear in $\ve$ perturbation of $\psi_0$, uniquely determined by the boundary perturbations $(g,h)$, together with a correction term that is $\ve^2$ small.

\begin{theorem}\label{thm:defPsive}
Let $(g,h)\in (C^{2,\alpha}(\mathbb{T}))^2$. There exists $\ve_0>0$ such that for all $0<\ve<\ve_0$ there exists a unique solution $\psi_\ve:D_\ve\rightarrow \R$ to \eqref{eq:steadyEulerDve} in $D_\ve$. Moreover, there exist functions $\varphi,\,r_\ve :D_\ve\rightarrow \R$ with $\Vert {\varphi} \Vert_{C^{2,\alpha}(D_\ve)}\lesssim 1$ and $\Vert r_\ve \Vert_{C^{2,\alpha}(D_\ve)}\lesssim \ve^2$ such that
\begin{equation}\label{eq:defPsive}
\psi_\ve(x,y) = {\psi_0}(x,y) + \ve {\varphi}(x,y) + r_\ve(x,y).
\end{equation}
\end{theorem}

The proof of Theorem \ref{thm:defPsive} consists of finding the solution $\psi_\ve$ to \eqref{eq:steadyEulerDve} as a perturbation of $\psi_0$ given by a precise first order approximation together with a correction that is $\ve^2$ small. The first order approximation is given by the unique solution $\varphi(x,y):D_0\rightarrow \R$ to
\begin{equation}\label{eq:HGvarphi}
\begin{cases}
\left(\Delta-F'(\psi_0)\right)\varphi = 0 & \text{ in } D_0 \\
\varphi(x,H(x))= -h(x)\partial_y\psi_0(x,H(x)) & \text{ for } x\in \mathbb{T}, \\
\varphi(x,G(x))= -g(x)\partial_y\psi_0(x,G(x)) & \text{ for } x\in\mathbb{T}. \\
\end{cases}
\end{equation} 
To define $\psi_\ve$ as in \eqref{eq:defPsive}, since $D_\ve$ may not be contained in $D_0$,
we first extend $\psi_0$ and $\varphi$ 
from $D_0$ to a slightly larger region containing both $D_0$ and $D_\ve$.  Since $G,H\in C^{2,\alpha}(\mathbb{T})$, this can be done using a standard extension operator which may be found in \cite[Lemma 6.37]{gilbargtrudinger}. We fix an open set $\widetilde{D_0} \Supset \overline{D_0}$ so that for $\ve$ small enough,
$\overline{D_\ve} \Subset {\widetilde{D}_0}$. There are 
extensions $\widetilde{\psi}_0$ and $\widetilde{\varphi}$ of $\psi_0$ and $\varphi$ defined on
$\widetilde{D}_0$, coinciding with $\psi_0$ and $\varphi$ on $D_0$, so that  $\Vert\widetilde{\psi}_0\Vert_{C^3(\widetilde{D_0})}\lesssim \Vert \psi_0 \Vert_{C^3(D_0)}$ and $\Vert \widetilde{\varphi}\Vert_{C^3(\widetilde{D_0})} \lesssim \Vert \varphi \Vert_{C^3(D_0)}$. In particular, they are defined on $D_\ve$. To ease notation, we drop the tildes and we assume $\psi_0$ and $\varphi$ already denote the extensions. As for the error function $r_\ve(x,y)$, it is defined to be the unique solution to
\begin{equation}\label{eq:rve}
\begin{cases}
\left(\Delta-F'(\psi_0)\right)r_\ve = \widetilde{\N}(r_\ve)(x,y) & \text{ in } D_\ve, \\
r_\ve(x,H(x))= r_\ve^{top}(x) & \text{ for } x\in \mathbb{T}, \\
r_\ve(x, G(x))= r_\ve^{bot}(x) & \text{ for } x\in\mathbb{T}, \\
\end{cases}
\end{equation}
where
\begin{align*}
\widetilde{\N}(v) = F(\psi_0 + \ve\varphi + v) - F'(\psi_0)r_\ve -\Delta\psi_0 - \ve\Delta \varphi
\end{align*}
and
\begin{align*}
r_\ve^{top}(x)&:= c_H-  \psi_0(x,H(x)+\ve h(x)) - \ve  \varphi(x,H(x)+\ve h(x)), \\
r_\ve^{bot}(x)&:= c_G-  \psi_0(x,G(x)+\ve g(x))- \ve\varphi(x,G(x)+\ve g(x)) 
\end{align*}
With this, we crucially note that $\psi_\ve(x,y):=\psi_0(x,y) + \ve\varphi(x,y) + r_\ve(x,y)$ is the unique solution to \eqref{eq:steadyEulerDve} if and only if $r_\ve$ is the unique solution to \eqref{eq:rve}. Then, Theorem \ref{thm:defPsive} follows directly from the above and
\begin{proposition}\label{prop:boundsrve}
Let $\alpha\in (0,1)$. Let $g,h\in C^{2,\alpha}(\mathbb{T})$ and $F\in C^4$. Then, there $\ve_0>0$ such that for all $0<\ve < \ve_0$, there exists a unique solution $r_\ve$ to \eqref{eq:rve} and it is such that for all $0 < \ve < \ve_*$,
\begin{align*}
\Vert r_\ve \Vert_{C^{2,\alpha}(D_\ve)} \lesssim \ve^2 \left( \Vert F \Vert_{C^4} + \Vert h \Vert_{C^{2,\alpha}} + \Vert g \Vert_{C^{2,\alpha}} \right).
\end{align*}
\end{proposition} 
The proof is straightforward and follows from a fixed point argument. We postpone it to Appendix \ref{estappend}. We remark here that while the required regularity for $F$ so that the above proposition holds true is $F\in C^{3,\alpha}$, we assume $F\in C^4$ for the sake of simplicity.

\subsection{Generic appearance of islands}

We first
show that any singular streamline $\Gamma_\ve$ of $\psi_\ve$ on which $\nabla \psi_\ve = 0$ must be $\ve$ close to the singular streamline $\Gamma_0$ of $\psi_0$. 

\begin{lemma}\label{lemma:ApproxGammave}
For $\ve>0$ small enough, any singular streamline $\Gamma_\ve$ of $\psi_\ve$ is given by a graph in $x$, that is, there exists some $y_\ve\in C^1(\mathbb{T})$ such that
\begin{align*}
\Gamma_\ve = \lbrace (x,y_\ve(x)): \, x\in \mathbb{T} \rbrace.
\end{align*}
Moreover, there holds $\Vert y_\ve - y_0 \Vert_{C^1(\mathbb{T})}\lesssim \ve$.
\end{lemma}

\begin{proof}
We begin by recalling that $F(c_0)=\Delta\psi_0|_{\Gamma_0}= \partial_y^2(x,y_0(x))\left( 1 + (y_0'(x))^2\right)\neq 0$ and assume without loss of generality that $F_0:=F(c_0)<0$, so that $\partial_y^2\psi_0(x,y_0(x))<0$ for all $x\in \mathbb{T}$. By uniform continuity, we have that $\partial_y\psi_0(x,y)\leq \frac{F_0}{2}$ for all $x\in \mathbb{T}$ and all $|y-y_0(x)|<\delta_0$, for some $\delta_0>0$ sufficiently small. In particular, since
\begin{align*}
\partial_y\psi_0(x,y_0(x) + \delta) = \delta\int_0^1 \partial_y^2\psi_0(x,y_0(x) + s\delta)\rmd s ,
\end{align*} 
for all $\delta\in \R$, we conclude that $\partial_y\psi_0(x,y_0(x)+\delta) \leq \delta \frac{F_0}{2} <0$ and $\partial_y^2\psi_0(x,y_0(x) - \delta) > -\delta\frac{F_0}{2}$, for all $0 < \delta < \delta_0$. Hence, for any fixed $\delta\in (0,\delta_0)$ we observe that
\begin{align*}
\partial_y\psi_{\ve}(x,y_0(x) +\delta) &= \partial_y\psi_0(x,y_0(x) +\delta) +\left( \ve\partial_y\varphi + \partial_y r_{\ve} \right)(x,y_0(x) +\delta) \\&
\leq \delta\frac{ F_0}{2} - \ve \Vert \varphi\Vert_{C^1(D_\ve)} - \Vert r_{\ve} \Vert_{C^1(D_\ve)} <0
\end{align*}
for $\ve>0$ sufficiently small, 
where $r_\ve$ is a remainder defined in \eqref{eq:defPsive}. Likewise, we have $\partial_y \psi_{\ve}(x,y_0(x) -\delta) >0$. In particular, this shows that for all $\ve>0$ small enough, for all $x\in\mathbb{T}$, there exists $y_{\ve}(x)\in (y_0(x)-\delta, y_0(x) + \delta)$ such that $\partial_y\psi_{\ve}(x,y_{\ve}(x))=0$. Since $\partial_y^2\psi_{\ve}(x,y) =\partial_y^2\psi_0(x,y_0(x)) +O(\ve) < 0$ for all $x\in \mathbb{T}$ and all $|y-y_0(x)|<\delta_0$, for $\ve$ sufficiently small, by the implicit function theorem we conclude that the mapping $x\mapsto (x,y_{\ve}(x))$ is a unique, $2\pi$-periodic and $C^1$ curve, with $\Vert y_\ve(x) - y_0(x) \Vert_{C(\mathbb{T})}\leq \delta$.

In fact, once we know $\Vert y_\ve(x) - y_0(x) \Vert_{C(\mathbb{T})}\leq \delta$ for all $\delta\in (0,\delta_0)$ by taking $\ve>0$ small enough, we can improve it to $\Vert y_\ve(x) - y_0(x) \Vert_{C^1(\mathbb{T})}\lesssim \ve$. Indeed, for all $x\in \mathbb{T}$ we have that
\begin{align}\label{eq:partialypsixyve}
0 &= \partial_y\psi_\ve(x,y_\ve(x)) = \partial_y \psi_0(x,y_\ve(x)) + \ve\partial_y\varphi(x,y_\ve(x)) + \partial_y r_\ve(x,y_\ve(x)) 
\end{align}
where further
\begin{align}\label{eq:partialypsi0xyve}
\partial_y \psi_0(x,y_\ve(x)) = (y_\ve(x) - y_0(x) ) \int_0^1 (\partial_y^2\psi_0)(x,y_0(x) + t(y_\ve(x) - y_0(x))) \rmd t.
\end{align}
Since $\partial_y^2\psi_0(x,y_0(x))\neq 0$, by uniform continuity we have that $\partial_y^2\psi_0(x,y)\neq 0$ for all $|y-y_0(x)|\leq \delta_0$, for some $\delta_0>0$. For $\ve$ small enough such that $\Vert y_\ve(x) - y_0(x) \Vert_{C(\mathbb{T})}\leq \delta_0$, we then have
\begin{align*}
\Vert y_\ve(x) - y_0(x) \Vert_{C(\mathbb{T})}\lesssim \ve \Vert \varphi \Vert_{C^1(D_{\ve})} + \Vert r_\ve \Vert_{C^1(D_{\ve})} \lesssim \ve.
\end{align*}
Similarly, applying $\frac{\rmd}{\rmd x}$ to \eqref{eq:partialypsixyve}  and \eqref{eq:partialypsi0xyve} we have that
\begin{align*}
0 &= (y_\ve(x) - y_0(x) )' \int_0^1 (\partial_y^2\psi_0)(x,y_0(x) + t(y_\ve(x) - y_0(x))) \rmd t \\  
&\quad + (y_\ve(x) - y_0(x) )'(y_\ve(x) - y_0(x) ) \int_0^1 (\partial_y^2\psi_0)(x,y_0(x) + t(y_\ve(x) - y_0(x))) \rmd t \\
&\quad + (y_\ve(x) - y_0(x) ) \int_0^1 (\partial_{xyy}^3\psi_0 + (y_0'(x))\partial_y^3\psi_0)(x,y_0(x) + t(y_\ve(x) - y_0(x))) \rmd t \\
&\quad + (y_\ve(x)-y_0(x))' \left( \ve \partial_y\varphi(x,y_\ve(x)) + \partial_y r_\ve(x,y_\ve(x)) \right) \\
&\quad + (\partial_x + y_0'(x)\partial_y)\left(\ve \partial_y\varphi(x,y_\ve(x)) + \partial_y r_\ve(x,y_\ve(x)) \right).
\end{align*}
Since $\Vert y_\ve - y_0 \Vert_{C(\mathbb{T})}\lesssim \ve$ and $\ve \Vert \varphi \Vert_{C^2(D_{\ve})} + \Vert r_\ve \Vert_{C^2(D_{\ve})}\lesssim \ve$, taking $\ve$ small enough we conclude that $\Vert y_\ve - y_0 \Vert_{C^1(\mathbb{T})}\lesssim \ve$. Finally, let $\Gamma_{\ve} := \lbrace (x, y_{\ve}(x)) : x\in \mathbb{T} \rbrace$ and note that $\partial_y\psi_{\ve}|_{\Gamma_{\ve}}=0$ and $\psi_\ve$ is periodic on $\Gamma_{\ve}$. Hence, there exists $x_0\in \mathbb{T}$ for which $\nabla\psi_{\ve}(x_0,y_{\ve}(x_0))=0$. If $\psi_{\ve}$ is laminar, then $\nabla\psi_{\ve}=0$ on $\Gamma_{\ve}$.
\end{proof}

We shall now see that if $\psi_\ve$ does not develop islands for any sufficiently small boundary perturbation $(G+\ve g,H+\ve h)$, the associated first order correction $\varphi=\varphi_{g,h}$ that solves \eqref{eq:HGvarphi} must be constant on the singular streamline $\Gamma_0$ of $\psi_0$ .

\begin{lemma}\label{lemma:laminarcontradiction}
Let $\lbrace \ve_n \rbrace_{n\geq 1}\subset (0,1)$ any sequence for which $\ve_n\rightarrow 0$ as $n\rightarrow \infty$ and $\psi_{\ve_n}$ does not have islands. Then $\varphi$ is constant on $\Gamma_0$.
\end{lemma}

\begin{proof}
Let $\ve>0$ and assume that $\psi_\ve$ does not have islands. Then it is laminar and by Lemma \ref{lemma:ApproxGammave} there exists a singular streamline $\Gamma_\ve= \lbrace (x,y_\ve(x)): \, x\in \mathbb{T} \rbrace$ of $\psi_\ve$, with $y_\ve\in C^1(\mathbb{T})$ and $\Vert y_\ve - y_0 \Vert_{C^1(\mathbb{T})}\lesssim \ve$, where $\Gamma_0= \lbrace (x,y_0(x)): \, x\in \mathbb{T} \rbrace$ denotes the unique singular streamline of $\psi_0$. Since $\psi_\ve$ is constant along $\Gamma_\ve$, we deduce that
\begin{align*}
0 = \frac{\rmd}{\rmd x}\psi_\ve(x,y_\ve(x)) = \frac{\rmd }{\rmd x} \left( \psi_0(x,y_\ve(x)) + \ve\varphi(x,y_\ve(x)) + r_\ve(x,y_\ve(x)) \right)
\end{align*}
where further
\begin{align*}
\psi_0(x,y_\ve(x)) &= \psi_0(x,y_0(x)) + \partial_y\psi_0(x,y_0(x))(y_\ve(x)-y_0(x)) \\
&\quad+ (y_\ve(x) - y_0(x) )^2 \int_0^1 \int_0^1 t(\partial_y^2\psi_0)(x,y_0(x) + st(y_\ve(x) - y_0(x))) \rmd s \rmd t
\end{align*}
and also
\begin{align*}
\varphi(x,y_\ve(x)) = \varphi(x,y_0(x)) + (y_\ve(x) - y_0(x))\int_0 ^1 (\partial_y\varphi)(x, y_0(x) + t(y_\ve(x) - y_0(x)) \rmd t.
\end{align*}
Since $\psi_0(x,y_0(x))$ is constant and $\partial_y\psi_0(x,y_0(x)) = 0$ for all $x\in \mathbb{T}$, together with $\Vert y_\ve -y_0 \Vert_{C^1(\mathbb{T})} \lesssim \ve$, we conclude that
\begin{align*}
0 = \ve \frac{\rmd }{\rmd x}\varphi(x,y_0(x)) + O(\ve^2)
\end{align*}
Therefore, for any sequence $\ve_n\rightarrow 0$ for which $\psi_{\ve_n}$ does not have islands, we must have $\varphi$ constant along $\Gamma_0$.
\end{proof}

Once we have established that $\varphi$ being non-constant along $\Gamma_0$ gives rise to islands for $\psi_\ve$ for $\ve>0$ small enough, we next investigate the set of boundary perturbations $(g,h)$ to $(G,H)$ for which the solution $\varphi=\varphi_{g,h}$ to \eqref{eq:HGvarphi} is not constant on $\Gamma_0$. For this, given $(g,h),\, (G,H)\in (C^{2,\alpha}(\mathbb{T}))^2$, and the solution $\varphi_{g,h}$ to \eqref{eq:HGvarphi} posed in $D_{G,H}$ with boundary values $(g,h)$, we introduce
\begin{align*}
\mathcal{B}_{G,H} = \lbrace (g,h)\in (C^{2,\alpha}(\mathbb{T}))^2 : \varphi_{g,h}|_{\Gamma_0}\neq \text{const}\rbrace
\end{align*}
We show that any solution $\psi_0$ in $D_{G,H}$ is an arbitrarily small perturbation away from developing islands.

\begin{lemma}\label{lemma:BGHdense}
For all $(G,H)\in (C^{2,\alpha}(\mathbb{T}))^2$ the set $\mathcal{B}_{G,H}$ is open and dense.
\end{lemma}

\begin{proof}
The open statement on $\mathcal{B}_{G,H}$ follows from the fact that $\varphi_{g,h}|_{\Gamma_0}\neq \text{const}$ is an open condition and $\varphi_{g,h}$ has continuous dependence on $(g,h)$ due to classical Schauder estimates.

As for the density statement, since $\varphi_{g,h}$ is linear in the boundary values $(g,h)$, it is enough to show that $\mathcal{B}_{G,H}$ is non-empty. Namely, we shall exhibit some $(g,h)\in (C^{2,\alpha}(\mathbb{T}))^2$ such that $\varphi_{g,h}$ is not constant along $\Gamma_0$. In this direction, let $g=0$, $h_1=1$, $h_2=\cos(x)$ and assume that $(0,h_i)\not \in \mathcal{B}_{G,H}$, for $i=1,2$.  That is, $\varphi_{0,h_i}|_{\Gamma_0}=:\varphi_i$ for some constants $\varphi_i\in \R$. We assume $\varphi_i\neq 0$. Then, for $h(x)=\varphi_2-\varphi_1\cos(x)$ we have that $\varphi_{0,h}|_{\Gamma_0}=0$ and $\varphi_{0,h}(x,G(x))=0$ as well. Therefore, due to the elliptic stability of $F'$ we conclude that $\varphi_{0,h}=0$ on $D_{G,\Gamma_0}$ and thus $\nabla\varphi_{0,h}=0$ on $\Gamma_0$. Since $\varphi_{0,h}$ solves \eqref{eq:HGvarphi}, by unique continuation (see \cite{nadirashvili1986uniqueness}) we have $\varphi_{0,h}=0$ in all $D_{G,H}$ and thus by continuity $\varphi_{0,h}(x,H(x))=0$, a contradiction with $\varphi_{0,h}(x,H(x)) = -\partial_y\psi_0(x,H(x))h(x)\neq 0$, since $\partial_y\psi_0(x,H(x))\neq 0$ due to Lemma \ref{lemma:nonzerofluidvel} and $h(x)\neq 0$. If $\varphi_i=0$ for some $i=1,2$, we argue for $h(x)=h_i(x)$ and we likewise reach $h_i(x)=0$. Hence, we must have either $(0,h_1)\in \mathcal{B}_{G,H}$ or $(0,h_2)\in \mathcal{B}_{G,H}$ and therefore we conclude that $\mathcal{B}_{G,H}$ is non-empty.
\end{proof}

\subsection{Genericity of Islands: Proof of Theorem \ref{thm:genericityisland}} Define the set
\begin{equation}
\mathcal{B} = \lbrace (G,H)\in (C^{2,\alpha}(\mathbb{T}))^2 : \text{ the solution }\psi_0 \text{ to } \eqref{eq:steadyEulerDve} \text{ in } D_{G,H} \text{ possess islands} \rbrace.
\end{equation}
We claim that it
is open and dense. The density of $\mathcal{B}$ follows from Theorem \ref{thm:defPsive}, Lemma \ref{lemma:laminarcontradiction} and Lemma \ref{lemma:BGHdense}, since there is an open and dense set of boundary perturbations $(G+\ve g,H + \ve h)$ to $(G,H)$ such that the associated solution $\psi_\ve$ develops islands. Therefore, the theorem is proved once we show that $\mathcal{B}$ is open in $(C^{2,\alpha}(\mathbb{T}))^2$. Fixing $(H,G)\in \mathcal{B}$, we have the following. 

\begin{proposition}\label{islandsopen}
There exists $\ve_0>0$ such that for all $(h,g)\in (C^{2,\alpha}(\mathbb{T}))^2$ with $\Vert g \Vert_{C^{2,\alpha}(\mathbb{T})}+\Vert h \Vert_{C^{2,\alpha}(\mathbb{T})}\leq 1$, the solution $\psi_\ve$ in $D_{G+\ve g, H+\ve h}$ contains closed contractible streamlines, for all $0<\ve < \ve_0$.
\end{proposition}

\begin{proof}
Let $\Gamma$ denote a closed contractible streamline of $\psi_0$ and let $c=\psi_0|_\Gamma$. The proof is divided according to the nature of the contractible streamline $\Gamma$.
\subsubsection*{Case 1} Assume that $\nabla\psi_0\neq 0$ on $\Gamma$.  Since $\psi_0$ is $C^2$, we can locally parametrize $\Gamma$ by a $C^1$ curve. Let $\hat{n}$ denote the normal unit vector pointing inside the set bounded by $\Gamma$ and assume without loss of generality) that $\hat{n}(\x)=\mathbf{e}_2$ and $\partial_{\hat{n}}\psi_0(\x)=\partial_y\psi_0(\x) > 0$ for some $\x\in \Gamma$. 

By continuity, $\frac12\partial_y\psi_0(\x) < \partial_y\psi_0(\x+y\hat{n}) < \frac32\partial_y\psi_0(\x)$, for all $0 < y < \delta_0$, for some $\delta_0$ small enough. Let $0 < 64y_1 < y_2 < \delta_0$ and observe that 
\begin{align*}
\psi_0(\x+y_1\hat{n}) < \psi_0(\x) + \frac32\partial_y\psi_0(\x) y_1 < \psi_0(\x + y\hat{n}) < \psi_0(\x) + \frac12\partial_y\psi_0(\x) y_2 < \psi_0(\x+y_2\hat{n})
\end{align*}
for all $y\in (64y_1,y_2)$. For $\Gamma_i = \lbrace \psi_0=\psi_0(\x+y_i\hat{n})\rbrace \ni \x_0+y_i\hat{n} $, for $i=1,2$, we see $\Gamma_y = \lbrace \psi_0=\psi_0(\x+y\hat{n})\rbrace \ni \x_0+y\hat{n}$ lies between $\Gamma_1$ and $\Gamma_2$, two closed and contractible streamline.  Moreover, 
\begin{align*}
 \psi_0(\x) + \frac32\partial_y\psi_0(\x) y_1 < \psi_\ve(\x + y\hat{n}) < \psi_0(\x) + \frac12\partial_y\psi_0(\x) y_2 
\end{align*}
for all $\ve>0$ small enough and hence $\Gamma_y^\ve = \lbrace \psi_\ve=\psi_\ve(\x+y\hat{n})\rbrace \ni \x_0+y\hat{n}$ still lies between $\Gamma_1$ and $\Gamma_2$, it is thus a closed contractible stream-line and $\psi_\ve$ has islands.

\subsubsection*{Case 2} Assume now $\nabla\psi_0=0$ somewhere in $\Gamma$. If there exists $\x\in \Gamma$ such that $\nabla\psi_0(\x)\neq 0$, by continuity there exists points inside the region bounded by $\Gamma$ where $\nabla\psi_0\neq 0$. At least one of them belongs to a regular streamline where $\nabla\psi_0\neq 0$ throughout (see \cite[Lemma 2.1]{drivas2024geometric}) and Case 1 applies.

\subsubsection*{Case 3} Assume next that $\nabla\psi_0=0$ and also $F(c)\neq 0$. Then, we can locally parametrize $\Gamma$ by a $C^1$ curve, see \cite[Lemma 2.3]{drivas2024geometric} and since $\Gamma$ is closed and contractible, there exists some $\x\in \Gamma$ such that $\hat{n}=\mathbf{e}_2$. Further assuming $F(c)>0$, we have that $\partial_y^2\psi_0(\x)>0$ and hence $\psi_0>c$ in a small annular region bounded by $\Gamma$ and $\psi_0$ is monotone increasing there. Thus, we proceed as in Case 1 to show the existence of two closed and contractible stream-lines $\Gamma_i$ of $\psi_0$ that enclose contractible streamlines of $\psi_\ve$, for $\ve>0$ small enough.

\subsubsection*{Case 4} We last consider $\nabla\psi_0=0$ on $\Gamma$ and $ F(c)=0$. Then, we can extend $\psi_0=c$ in the region enclosed by $\Gamma$ and the resulting function still satisfies \eqref{eq:steadyEulerDve} in $D_{G,H}$. Since $F$ is $C^2$, the unique continuation principle, see \cite{nadirashvili1986uniqueness}, shows that $\psi_0=c$ in all $D_{G,H}$, thus contradicting our assumptions.
\end{proof}

\subsection{Islands in generic annular domains}
We briefly comment on the setting where the Euler equations are considered in an annular domain $D_{\Gamma_g,\Gamma_h}$  of $\mathbb{T}\times\R$ defined to be the region between  two closed non-contractible  $C^{2,\alpha}$ curves $\Gamma_h$ and $\Gamma_g$. As before, we consider solutions $\psi$ to  
\begin{equation}\label{eq:steadyEulerDannular}
\begin{cases}
\Delta\psi = F(\psi) & \text{ in } D_{\Gamma_g,\Gamma_h}, \\
\psi=c_h & \text{ on } \Gamma_h, \\
\psi=c_g & \text{ on } \Gamma_g,
\end{cases}
\end{equation}
for any $F\in C^1$, and $c_g,c_h\in \R$. We fix $\alpha\in (0,1)$ and $F\in C^4(\mathbb{R};\mathbb{R})$ satisfying  $F'>-\lambda_1$, where $\lambda_1$ is the first eigenvalue of the Dirichlet Laplacian in $D_{\Gamma_g,\Gamma_h}$. Theorem \ref{thm:genericityisland} now reads
\begin{theorem}\label{thm:annulargenericityisland}
There exists an open dense set $\B$ in the product space of $C^{2,\alpha}$ curves so that for each $(\Gamma_g,\Gamma_h)\in \B$,  any solution $\psi$ to \eqref{eq:steadyEulerDannular}  on $D_{\Gamma_g,\Gamma_h}$  having stagnation points possess islands.
\end{theorem}

The program to prove Theorem \ref{thm:annulargenericityisland} follows closely that of Theorem \ref{thm:genericityisland}. It involves a precise understanding of any potential singular streamline $\Gamma_0$ of $\psi_0$. By Lemma \ref{lemma:nonzeroFc0}, $\Gamma_0$ is a closed non-contractible curve and the previous arguments of the section apply, showing that the solution $\psi_\ve$ associated to any boundary perturbation can be understood as the perturbation of the base solution $\psi_0$ by a first order linear term and a second order correction. As before, the first order correction of nearby solutions satisfies a linear elliptic equation with boundary data linearly determines by the boundary perturbations and it must be constant on the singular streamline $\Gamma_0$ of $\psi_0$ to prevent the appearance of islands. Then, one can argues as in Lemma \ref{lemma:BGHdense} to show this is not the case for an open dense set of boundary perturbations, we omit the routine details.

We remark again that in general the complement of $\mathcal{B}$ is not empty. Indeed,   Cauchy-Kovalevskaya provides examples of a steady laminar Euler flow that solves \eqref{eq:steadyEulerDannular} for $F$ real analytic in a compact doubly connected domain $D_{\Gamma_g,\Gamma_h}$ with non-circular boundaries $\Gamma_g$ and $\Gamma_h$ and zero velocity on $\Gamma_h$.  Also \cite[Theorem 4.12]{elgindi2024classification}  for a particular lower regularity $F$.

\section{Islands for small perturbations of the periodic channel}\label{sec:pert}
In this section we further investigate the appearance of islands for small perturbations of $D_0=\mathbb{T}\times(-1,1)$.  We first assume that $c_h=c_g=0$ (which just amounts to a possible shift in $F$ if $c_h=c_g\neq 0$).  This is the assumption of trivial homology which was studied in \cite{drivas2024islands}. Under this assumption, we can relax the assumption on $F$ so that the problem
\begin{equation}\label{eq:linearpdeF'}
\begin{cases}
\Delta\varphi - F'(\psi) \varphi = 0 & \text{ in } \mathbb{T}\times[-1,1], \\
\varphi(x,1)=0 & \text{ for } x\in\mathbb{T}, \\
\varphi(x,-1)=0 & \text{ for } x\in\mathbb{T}.
\end{cases}
\end{equation}
 admits only the trivial solution. In the end of the section we show how to treat the general case $c_h\neq c_g$. 
Let $\psi_0(x,y):\mathbb{T}\times[-1,1]\rightarrow\R$ be any non-zero solution to 
\begin{equation}\label{eq:Psi0}
\begin{cases}
\Delta\psi_0 = F(\psi_0) & \text{ in } \mathbb{T}\times[-1,1], \\
\psi_0(x,\pm 1)=0 & \text{ for } x\in\mathbb{T}.
\end{cases}
\end{equation}
Note that the existence of a critical point for $\psi_0$ is already ensured by the trivial homology.  
Since $\partial_x\psi_0$ is a solution to the linearised problem, which admits only the trivial solution,  we obtain that $\psi_0(x,y) = \psi_0(y)$ is a shear flow, and from Theorem 2.2 in \cite{bcn97} we deduce that $\psi_0(y)$ is symmetric with respect to $y=0$ and thus $\psi_0'(0)=0$. Moreover, there holds $\psi_0'(1)\neq 0$. Otherwise,  $\Phi_0(y):=\psi_0'(y)$ satisfies
\begin{equation}\label{eq:Phi0}
\begin{cases}
\Phi_0'' = F'(\psi_0)\Phi_0 & \text{ in } [-1,1], \\
\Phi_0(\pm 1)=0, & 
\end{cases}
\end{equation}
for which $\Phi_0\equiv 0$ and hence $\psi_0\equiv 0$, a contradiction. Using the description of $\psi_\ve$ from Theorem \ref{thm:defPsive}, the condition on $\varphi_{g,h}$ from Lemma \ref{lemma:laminarcontradiction} now reads $\varphi(x,0)$ is constant, since $\Gamma_0 = \lbrace y = 0\rbrace$. The next result determines  $\varphi(x,0)$ in terms of the boundary defining functions $(g,h)$. More precisely, for
\begin{equation}\label{eq:defBprime}
\mathcal{B}'=\lbrace (g,h)\in (C^2(\mathbb{T}))^2 : h'(x)+g'(x)\not\equiv 0 \rbrace,
\end{equation}
we have the following.

\begin{lemma}\label{lemma:pxvarphi0}
Let $(g,h)\in (C^{2,\alpha}(\mathbb{T}))^2$ and let $\varphi(x,y)$ be the solution to \eqref{eq:HGvarphi} in $D_0=\mathbb{T}\times(-1,1)$. Then, $\partial_x\varphi(x,0)\not \equiv 0$ if and only if $(g,h)\in \B'$.
\end{lemma}

\begin{proof}
Firstly, we decompose $\varphi(x,y) = \varphi_e(x,y) + \varphi_o(x,y)$ where $\varphi_e$ and $\varphi_o$ denote the even symmetric and odd symmetric parts of $\varphi$ with respect to $y=0$. By definition, there clearly holds $\varphi_o(x,0) = 0$ and $\partial_y\varphi_e(x,0) = 0$, for all $x\in\mathbb{T}$ so that, in particular, $\varphi(x,0) = \varphi_e(x,0)$. Thus, we study $\varphi_e$, which is the unique solution to
 \begin{align}\label{eq:ellipticevenF}
\begin{cases}
\Delta\varphi_e-F'(\psi_0)\varphi_e= 0 & \text{ in } \mathbb{T}\times[0,1], \\
\varphi_e(x,1)= f(x) & \text{ for } x\in \mathbb{T}, \\
\partial_y\varphi_e(x,0)= 0 & \text{ for } x\in\mathbb{T}. \\
\end{cases}
\end{align}
for $f(x) = -\psi_0'(1)\frac{h(x)+ g(x)}{2}$, because $\psi_0$ is even. Now, if $\partial_x\varphi_e(x,0)=0$, for all $x\in \mathbb{T}$, let $\phi:=\partial_x\varphi$, which satisfies
 \begin{align}\label{eq:ellipticphi}
\begin{cases}
\Delta\phi-F'(\psi_0)\phi= 0 & \text{ in } \mathbb{T}\times[0,1], \\
\phi(x,0)=0 & \text{ for } x\in \mathbb{T}, \\
\partial_y\phi(x,0)= 0 & \text{ for } x\in\mathbb{T}. \\
\end{cases}
\end{align}
From Lemma 3 in \cite{drivas2024islands}, we have $\phi=0$ in $\mathbb{T}\times[0,1]$ and thus $\partial_x\varphi_e(x,1)=0$. Since $\psi_0'(1)\neq 0$, we deduce that $f(x)$ must be constant. In other words,
\begin{align}\label{eq:islandcondition}
h(x) + g(x) = c_0, \quad \text{ for all }x\in \mathbb{T},
\end{align}
for some $c_0\in \R$. With this, the lemma is proved.
\end{proof}

\begin{remark}[Possible existence of laminar flow]

It is not clear either whether there exist laminar solutions $\psi$ of \eqref{eq:steadyEulerDve}, even if one takes $h$ and $g$ to be $\ve$ small. In this direction, Lemma \ref{lemma:laminarcontradiction} requires $h(x)+g(x) = f_0$, for some $f_0\in \R$ for any such laminar solution to exist.  We ask:
\begin{question}
For all $\ve>0$ sufficiently small does there exist $g,h:\mathbb{T}\rightarrow \R$ and $F$ such that the solution $\psi$ to \eqref{eq:steadyEulerDve} posed in $D_{-1-\ve g,1+\ve h}$ does not have islands?
\end{question}
Supporting evidence can be found in the case of solutions close to the Couette flow, which solves \eqref{eq:steadyEulerDve} posed in $\mathbb{T}\times(-1,1)$ with $F=-1$ and  $\psi_0 = -\frac12(y^2-1)$. Then, the equation for the first order correction  now reads
\begin{equation*}
\begin{cases}
\Delta\varphi = 0 & \text{ in } \mathbb{T}\times [-1,1], \\
\varphi(x,1)=h(x), & x\in\mathbb{T} \\
\varphi(x,-1)= g(x), & x\in \mathbb{T}
\end{cases}
\end{equation*}
whose solution is 
\begin{align*}
\varphi(x,y) = \sum_{n\neq 0} \frac{1}{2\sinh(n)}\left( h_n - g_n \right) \sinh(ny)e^{inx} + \sum_{n\in \Z} \frac{1}{2\cosh(n)}\left( h_n + g_n \right) \cosh(ny)e^{inx}
\end{align*}
where $h(x) = \sum_{n\in \Z} h_ne^{inx},$ and  $g(x) = \sum_{n\in \Z} g_ne^{inx}.$
It is instructive to note that for any singular streamline $\Gamma_\ve$ to exist, we must have $h_n+g_n=0$ for all $n\neq 0$ and also
\begin{align*}
y_\ve(x) &= \ve \partial_y\varphi(x,y_\ve(x)) + \partial_yr_\ve(x,y_\ve(x)) = \ve \partial_y\varphi(x,0) + O(\ve^2)
\end{align*}
so that $\frac{y_\ve(x)}{\ve}\rightarrow \partial_y\varphi(x,0)$ as $\ve\rightarrow 0$. In other words, the limiting rescaled $\Gamma_\ve$ converges to the curve $(x,\partial_y\varphi(x,0))$, which is uniquely prescribed by the boundary data $(g,h)$ by 
\begin{align*}
\partial_y\varphi(x,0) = \sum_{n\neq0 } \frac{nh_n}{\sinh(n)}e^{inx}
\end{align*}
and clearly indicates there is a unique correspondence between boundary data and any potentially singular streamline. It is thus conceivable that there are non-flat domains which support laminar flows having stagnation sets, but we do not resolve this point. {It is worth mentioning here \cite[Theorem 4.12]{elgindi2024classification}, which shows the existence of Arnold stable steady laminar solutions of the Euler equations on narrow compact doubly connected non-circular domains that stagnate along the outer boundary and satisfy \eqref{eq:steadyEulerDannular} for some non-linearity $F\in C^{2,\frac12}$. It would be interesting to see if the solutions constructed in \cite{elgindi2024classification} can be extended to occupy larger domains.}
\end{remark}

\subsection{Islands for non-trivial homology}
Here we show Theorem \ref{thm:genericityisland} for the case $c_h\neq c_g$. Since $\psi_0(x,1)=c_h$ and $\psi_0(x,-1)=c_g$, the assumed elliptic stability of $F'(\psi_0)$ gives $\psi_0(x,y)=\psi_0(y)$, as in the trivial homology setting. Hence, by the stagnation assumption, there exists some $y_0\in (-1,1)$ such that $\psi_0'(y_0)=0$. We claim that such $y_0$ is unique.

Indeed, let $y_1\in [-1,1]$, $y_1\neq y_0$ be any other critical point. Assuming, say, that $y_1>y_0$, we have that $\Phi_0(y) = \psi_0'(y)$ satisfies
\begin{equation}
\begin{cases}
\Phi_0'' = F'(\psi_0)\Phi_0 & \text{ in } [y_0,y_1], \\
\Phi_0(y_0)=\Phi(y_1)=0, & 
\end{cases}
\end{equation}
Then again, the elliptic stability of $F'$ shows that $\Phi_0(y)=0$ for all $y\in [y_0,y_1]$. In particular, this shows that $\psi_0(y) = \psi_0(y_0)$ and $F(\psi_0(y_0))=0$, for all $y\in [y_0,y_1]$. Hence, $\psi_0(y)=\psi_0(y_0)$, for all $y\in [-1,1]$, but this contradicts $c_g\neq c_h$.

We next observe that $\psi_0$ must be symmetric about $y=y_0$. Indeed, since $F(\psi_0(y_0))\neq 0$ and $\psi_0$ has only one critical point, then $\psi_0$ must attain either $c_h$ or $c_g$ only once inside. Say $\psi_0$ attains $c_h$ at some $y_1\in (-1,1)$. Then, in $(y_1,1)$ we have
\begin{equation}
\begin{cases}
\psi_0'' = F(\psi_0) & \text{ in } [y_1,1], \\
\psi_0(y_1)=\psi(1)=c_h, & 
\end{cases}
\end{equation}
and from \cite[Theorem 2.2]{bcn97} we deduce that $\psi_0$ must be symmetric with respect to $y_*=\frac{1+y_1}{2}$, with $\psi_0'(y_*)=0$. By uniqueness, it must be $y_*=y_0$ so that $y_1=2y_0-1$. Next, we argue as in the trivial homology setting, so that the analogue of Theorem \ref{thm:defPsive} holds and we have that for all $(g,h)\in (C^{2,\alpha}(\mathbb{T}))^2$. the unique solution $\psi_\ve:D_\ve\rightarrow \R$ to \eqref{eq:steadyEulerDve} in $D_\ve$ can be written as
\begin{equation}
\psi_\ve(x,y) = {\psi_0}(y) + \ve {\varphi}(x,y) + r_\ve(x,y),
\end{equation}
where $\varphi(x,y)$ solves
\begin{equation}\label{eq:varphihomo}
\begin{cases}
\left(\Delta-F'(\psi_0)\right)\varphi = 0 & \text{ in } \mathbb{T}\times[-1,1], \\
\varphi(x,1)= -h(x)\psi_0'(1) & \text{ for } x\in \mathbb{T}, \\
\varphi(x,-1)= g(x)\psi_0'(-1) & \text{ for } x\in\mathbb{T}, \\
\end{cases}
\end{equation}
with $\Vert {\varphi} \Vert_{C^{2,\alpha}(D_\ve)}\lesssim 1$ and $\Vert r_\ve \Vert_{C^{2,\alpha}(D_\ve)}\lesssim \ve^2$.  The obvious modifications to Lemma \ref{lemma:laminarcontradiction} show that the presence of islands for $\psi_\ve$ for all $\ve>0$ small enough is conditional on $\varphi(x,y_0)$ being non-constant. Then, Theorem \ref{thm:genericityisland} is established once we prove
\begin{lemma}\label{lemma:genericityislandnontrivialhom}
There exists an open and dense set $\mathcal{B}\subset (C^{2,\alpha}(\mathbb{T}))^2$ such that for all $(g,h)\in \mathcal{B}$, the solution $\varphi(x,y)$ associated to \eqref{eq:varphihomo} is non-constant along $y=y_0$.
\end{lemma}

\begin{proof}
    The open statement is straightforward due to classical Schauder estimates.  As for the density statement, we observe that solutions $\varphi$ to  \eqref{eq:varphihomo} are linear in $(g,h)$. Thus, it is enough to show there exists $(g,h)\in (C^{2,\alpha}(\mathbb{T}))^2$ such that $\varphi(x,y_0)$ is non-constant.  In this direction, let $g(x)\equiv g_0$ for some $g_0\in \R$ and $h'(x)\neq 0$. We claim that $\partial_x\varphi(x,y_0)\neq 0$. 

Assume otherwise. Then, $\partial_x\varphi(x,y)$ satisfies
\begin{equation}
\begin{cases}
\left(\Delta-F'(\psi_0)\right)\partial_x\varphi = 0 & \text{ in } \mathbb{T}\times[-1,y_0], \\
\partial_x\varphi(x,y_0)= 0 & \text{ for } x\in \mathbb{T}, \\
\partial_x\varphi(x,-1)= 0 & \text{ for } x\in\mathbb{T}, \\
\end{cases}
\end{equation}
so that $\partial_x\varphi(x,y)=0$, for all $(x,y)\in \mathbb{T}\times[-1,y_0]$ owing to the elliptic stability of $F'$. On the other hand, since $\varphi(x,y_0)$ is constant, we can consider the even and odd reflections of $\varphi$ with respect to $y=y_0$ in the domain $\mathbb{T}\times(2y_0-1,1)$. As usual, $\varphi_o(x,y_0)$ is trivially zero and thus we have $\partial_x\varphi_e(x,y_0)=0$ and also $\partial_x\partial_y\varphi_e(x,y_0)=0$. Again using Lemma 3 from \cite{drivas2024islands}, we conclude that $\partial_x\varphi_e(x,y)\equiv 0$ on $\mathbb{T}\times(y_0,1)$ and in particular this means that $\partial_x\varphi_e(x,1)=0$. However, 
\begin{align*}
\varphi_e(x,1) = \frac{-\psi_0'(1)h(x) + \varphi(x,2y_0-1)}{2}
\end{align*}
and $-1<2y_0-1<y_0$. Since $\partial_x\varphi(x,2y_0-1) = 0$, we conclude that $\partial_x\varphi(x,1) = -\frac{\psi_0'(1)}{2}h'(x)\neq 0$, for any non-constant $h$. Hence a contradiction and thus $\varphi(x,y_0)$ is not constant, namely $(g,h)\in \mathcal{B}$.
\end{proof}

\section{Size of Islands}\label{sec:sizeislands}
In this section we consider $D_0=\mathbb{T}\times(-1,1)$ and $D_\ve=D_{-1-\ve g, 1+\ve h}$, for some $(g,h)\in \mathcal{B}$ given by \eqref{eq:defBprime} or Lemma \ref{lemma:genericityislandnontrivialhom} according to the homology. We locate islands of fluid and characterise their size in terms of the boundary perturbations. To do so, we first identify points in $D_\ve$ for which islands appear nearby. For all $\ve>0$ small enough, the condition $F(\psi_0(s))\neq 0$ whenever $\psi_0'(s)=0$ provides the existence of some curve $\Gamma_\ve=\lbrace (x,y_\ve(x)):\, x\in \mathbb{T} \rbrace$ for which $\partial_y\psi_\ve=0$ on $\Gamma_\ve$.  Such curve is $\ve$ close to the critical point $y=y_0$ of $\psi_0$. Throughout the section, we assume that $F(\psi_0(y_0))<0$.

\begin{lemma}\label{lemma:maximizerlocation}
There exists $\delta_0>0$ such that $\psi_\ve:\mathbb{T}\times[y_0-\delta_0,y_0+\delta_0]\rightarrow \R$ attains its maximum in the interior. Moreover, its maximizers are close to global maximizers of $\varphi(\cdot,y_0)$, for $\ve>0$ small enough. 
\end{lemma}

\begin{proof}
Proceeding as in the proof of Lemma \ref{lemma:laminarcontradiction}, there exists $\delta_0>0$ small enough for which any global maximizer $(x_\ve,y_\ve)$ of $\psi_\ve$ in $\mathbb{T}\times[y_0-\delta_0,y_0+\delta_0]$ is interior, since $\partial_y\psi_\ve(x,y_0+\delta_0)<0$, $\partial_y\psi_\ve(x,y_0-\delta_0)>0$ and $\partial_y^2\psi_\ve(x,y)<0$ for all $(x,y)\in \mathbb{T}\times[y_0-\delta_0,y_0+\delta_0]$, for $\ve>0$ small enough. Hence, for any $x\in \mathbb{T}$, there exists a unique $y=y_\ve(x)\in (y_0-\delta_0,y_0+\delta_0)$ such that $\partial_y\psi_\ve(x,y_\ve(x))=0$. By the implicit function theorem, since $\partial_y^2\psi_\ve(x,y)<0$ in $\mathbb{T}\times[y_0-\delta_0,y_0+\delta_0]$ we see that $\Gamma_\ve = \lbrace (x,y_\ve(x)): \, x\in \mathbb{T}\rbrace$ is a $C^1$ curve in $\mathbb{T}\times[y_0-\delta_0,y_0+\delta_0]$.

Next, the map $x\mapsto \psi_\ve(x,y_\ve(x))$ is $2\pi$-periodic, and thus attains its maximum and minimum. Let $(x_\ve,y_\ve)\in \Gamma_\ve$ be any such maximum, there holds
\begin{align*}
\psi_\ve(x,y_\ve(x))=\psi_0(y_\ve(x)) + \ve \varphi(x,y_\ve(x)) + r_\ve (x,y_\ve(x)) \leq \psi_0(y_\ve) + \ve \varphi(x_\ve,y_\ve) + r_\ve (x_\ve,y_\ve) = \psi_\ve(x_\ve,y_\ve)
\end{align*}
for all $x\in\mathbb{T}$. As $\ve\rightarrow 0$ we have $x_\ve\rightarrow x_0$ up to a subsequence and then
\begin{align*}
\varphi(x,y_0)\leq \varphi(x_0,y_0) + O(|x_\ve - x_0|) + O(\ve)
\end{align*}
for all $x\in \mathbb{T}$. In particular, this shows that $x_0$ must be a global maximum of $\varphi(x,y_0)$. Therefore global maxima of $\psi_\ve$ in $\mathbb{T}\times[y_0-\delta_0,y_0+\delta_0]$ must converge to global maxima of $\varphi(\cdot,y_0)$. Indeed, if $x_0$ is not a global maximum of $\varphi(x,y_0)$, let $x_1$ denote a global maximum of $\varphi(x,y_0)$ and let $\ve>0$ small enough so that $O(|x_\ve - x_0|) + O(\ve) < \frac{\varphi(x_1,y_0) - \varphi(x_0,y_0)}{2}$. Then,
\begin{align*}
\varphi(x_1,y_0) \leq \varphi(x_0,y_0) + O(|x_\ve - x_0|) + O(\ve) < \frac12 \left( \varphi(x_1,y_0) + \varphi(x_0,y_0) \right)
\end{align*}
from which we readily obtain $\varphi(x_1,y_0)< \varphi(x_0,y_0)$, a contradiction.
\end{proof}

Given the preceding lemma, understanding the maximizers of  $\varphi(\cdot,y_0)$ and their possible degeneracy (whether $\partial_x^2\varphi$ vanishes) is crucial in describing $\psi_\ve(x,y)$ for $(x,y)$ nearby a maximum $(x_\ve,y_\ve)$.  For this, let $(g,h)\in (C^{2,\alpha}(\mathbb{T}))^2$ and let $\varphi(x,y):=\varphi_{g,h}(x,y)$ denote the unique solution to \eqref{eq:HGvarphi}. We define
\begin{equation}\label{eq:defB0}
\mathcal{B}_0 = \lbrace (g,h) \in (C^{2,\alpha}(\mathbb{T}))^2 : \max\varphi(\cdot,y_0)\neq 0 \text{ and }\partial_x^2\varphi(x_0,y_0)< 0 \text{ for all } (x_0, y_0)=\text{arg}\max\varphi(\cdot,y_0) \rbrace,
\end{equation}
and we readily observe that $\mathcal{B}_0\subset \mathcal{B}$.
\subsection{Local Morse theory}
Using Morse theory nearby suitable maxima of $\psi_\ve$, the next result shows that boundary perturbations $(g,h)\in \mathcal{B}_0$ of size $\ve$ generate islands of height $\ve^\frac12$.
\begin{proposition}\label{prop:islandsize}
Let $(g,h)\in \mathcal{B}_0$. There exists $\delta_1>0$ for which there exists $\ve_0>0$ such that for all $0< \ve < \ve_0$, the steady Euler solution $\psi_\ve$ to \eqref{eq:steadyEulerDve} possess islands whose height is comparable to $\ve^\frac12$. More precisely, there exists a local maximum $(x_\ve,y_\ve)$ of $\psi_\ve$ and constants $0<C_1\leq C_2$ for which for all $0< \delta< \delta_1$, the level set 
\begin{align*}
A_\delta:=\lbrace \psi_\ve-\psi_0(y_0) = (1-\delta)(\psi_\ve(x_\ve,y_\ve) -\psi_0(y_0) ) \rbrace
\end{align*}
is such that, for all $(x,y)\in A_\delta$,
\begin{align*}
C_1\delta \ve \leq (y-y_\ve)^2 + \ve (x-x_\ve)^2 \leq C_2\delta\ve.
\end{align*}
\end{proposition}
Since the proposition describes the level sets of $\Phi_\ve:=\psi_\ve-\psi_0(y_0)$ locally around $(x_\ve,y_\ve)$, the same conclusions can be drawn for the level sets of $\psi_\ve=\Phi_\ve + \psi_0(y_0)$ locally around$(x_\ve,y_\ve)$, since $\psi_0(y_0)$ is constant.

\begin{proof}
Let $\delta_0>0$ given by Lemma \ref{lemma:maximizerlocation} and let $(x_\ve,y_\ve)$ any global maximizer of $\psi_\ve(x,y_\ve(x))$ (and thus of $\psi_\ve$) in $\mathbb{T}\times[y_0-\delta_0,y_0+\delta_0]$. There holds $0=\frac{\rmd}{\rmd x}\big|_{x=x_\ve}\psi_\ve(x,y_\ve(x)) =\partial_x\psi_\ve(x_\ve,y_\ve) = 0$, so that $\nabla\psi_\ve(x_\ve, y_\ve) = 0$. Then, 
\begin{align*}
\psi_\ve(x,y) = \psi_\ve(x_\ve,y_\ve) + \int_0^1 (((x-x_\ve)\partial_x + (y-y_\ve)\partial_y)\psi_\ve) (x_\ve+t(x-x_\ve),y_\ve + t(y-y_\ve)) \rmd t
\end{align*}
and since $\nabla\psi_\ve(x_\ve,y_\ve)=0$, we further expand the above as
\begin{equation}
 \psi_\ve(x,y) = \psi_\ve(x_\ve,y_\ve) + (x-x_\ve)^2 H_{xx}(x,y) + 2(x-x_\ve)(y-y_\ve) H_{xy}(x,y)+ (y-y_\ve)^2 H_{yy}(x,y),
 \label{taylor}
\end{equation}
where the above functions are given explicitly by
\begin{equation}
 H_{x^ix^j}(x,y) = \int_{0}^1\int_0^1 t (\pa_{x^i}\pa_{x^j} \psi_\ve)(x_\ve + st(x-x_\ve), y_\ve + st(y-y_\ve))\, \rmd s \rmd t,
 \qquad
 x^1 = x, x^2 = y.
 \label{Hformula}
\end{equation}
As $F(\psi_0(y_0))<0$  we have $\frac34F(\psi_0(y_0)) < H_{yy} < \frac14 F(\psi_0(y_0))$ everywhere in the domain provided both $|y-y_\ve|$ and  $\ve$ are taken sufficiently small, since 
\begin{align*}
 H_{yy}(x,y) &= \int_0^1 \int_0^1  t \psi_0''(y_\ve +st(y-y_\ve)) \rmd s \rmd t \\
 &\quad+  \int_0^1 t \left( \ve\partial_y^2\varphi +\partial_y^2r_\ve \right)(x_\ve + st(x-x_\ve), y_\ve + st(y-y_\ve)) \rmd s \rmd t \\
 &= \frac{1}{2}F(\psi_0(y_0)) + O(|y-y_\ve|) + O(\ve)
\end{align*}
We also record that
\begin{equation}
 H_{xy} = \ve h_{xy},
 \qquad
 H_{xx} = \ve h_{xx},
 \label{h2formulas}
\end{equation}
for smooth functions $h_{xy}, h_{yy}$ uniformly bounded in $\ve$. Next, with
\begin{equation}
 w(x,y) =  (y-y_\ve)\sqrt{|H_{yy}(x,y)|} - (x-x_\ve) \frac{H_{xy}(x,y)}{\sqrt{|H_{yy}(x,y)|}},
\end{equation}
we note that $w(x,y) \sim y-y_\ve + O(\ve) (x-x_\ve)$ since $H_{xy} \sim \ve$ by \eqref{h2formulas}. Then, since $|H_{yy}| = -H_{yy}$, there holds
\begin{equation}
 w(x,y)^2 = -H_{yy}(x,y) (y-y_\ve)^2 - 2(x-x_\ve)(y-y_\ve) H_{xy}(x,y) - (x-x_\ve)^2 \frac{H_{xy}(x,y)^2}{H_{yy}(x,y)},
 \label{}
\end{equation}
and as a result, from \eqref{taylor} we have
\begin{align*}
\psi_\ve(x,y) &= \psi_\ve(x_\ve,y_\ve) - w(x,y)^2
 + (x-x_\ve)^2 \left( H_{xx}(x,y) - \frac{H_{xy}(x,y)^2}{H_{yy}(x,y)}\right)
 \\
 &= \psi_\ve(x_\ve,y_\ve) - w(x,y)^2 + \frac{ \mathrm{D}(x,y)}{H_{yy}(x,y)} (x-x_\ve)^2 \\
&\sim \psi_\ve(x_\ve,y_\ve) - w(x,y)^2 -  \mathrm{D}(x,y) (x-x_\ve)^2,
\end{align*}
where $\mathrm{D}(x,y) = \det H(x,y)$. Expanding as in \eqref{taylor}, we have
\begin{equation}
\mathrm{D}(x,y)= \mathrm{D}(x_\ve,y_\ve) +
 \int_{0}^1 ((x-x_\ve)\pa_x + (y-y_\ve) \pa_y) \mathrm{D}(x_\ve +t(x-x_\ve), y_\ve + t(y-y_\ve))\, dt.
 \label{Dexpansion}
\end{equation}
Now, from \eqref{h2formulas},
\begin{equation}
 \mathrm{D}(x,y) = \ve (h_{xx}H_{yy} - \ve h_{xy}^2),
\end{equation}
where $h_{xx}, H_{yy}, h_{xy}$ are uniformly bounded in $\ve$. On the other hand, since $\partial_y\psi_\ve(x,y_\ve(x))=0$ for all $x\in \mathbb{T}$, then there holds
\begin{align*}
0 = \partial_{xy}^2\psi_\ve(x,y_\ve(x)) + y'_\ve(x)\partial_y^2\psi_\ve(x,y_\ve(x)).
\end{align*}
and thus
\begin{align*}
\mathrm{D}(x_\ve,y_\ve) &= \partial_y^2\psi_\ve(x_\ve,y_\ve)\left( \partial_x^2 \psi_\ve(x_\ve,y_\ve) - (y_\ve'(x_\ve))^2\partial_y^2 \psi_\ve(x_\ve,y_\ve) \right) \\
&= \ve F(\psi_0(y_0))\partial_x^2\varphi(x_\ve,y_\ve) + O(\ve^2) \\
&= \ve F(\psi_0(y_0))\partial_x^2\varphi(x_\ve,y_0) + O(\ve)
\end{align*} 
Now, from Lemma \ref{lemma:maximizerlocation}, for $\ve>0$ sufficiently small, $x_\ve$ is (up to a subsequence) arbitrarily close to some maximizer $x_0$ of $\varphi(\cdot, y_0)$, so that
\begin{align*}
\mathrm{D}(x_\ve,y_\ve) &=\ve F(\psi_0(y_0))\partial_x^2\varphi(x_0,y_0) + O(|x_\ve-x_0|) + O(\ve).
\end{align*}
In particular, since $(g,h)\in \mathcal{B}_0$, we have that $\partial_x^2\varphi(x_0,y_0) < 0$ is independent of $\ve$ and thus we conclude that $\mathrm{D}(x_\ve,y_\ve) \sim \ve$. As a result, together with \eqref{Dexpansion} we find
\begin{equation}
 D(x,y) \sim \ve (1+ |x-x_\ve| + |y-y_\ve|).
 \label{}
\end{equation}

To properly describe the size of the island, let $\Phi_\ve(x,y) = \psi_\ve(x,y) - \psi_0(y_0)$ and recall that the connected component of $\lbrace \Phi_\ve(x,y) = \Phi_\ve(x_\ve,y_\ve) \rbrace$ containing $(x_\ve,y_\ve)$ consists only of the maximizer $(x_\ve,y_\ve)$ of $\Phi_\ve$.  Note that
\begin{align*}
\Phi_\ve(x_\ve,y_\ve) &= \psi_0(y_\ve) + \ve\varphi(x_\ve,y_\ve) + r_\ve(x_\ve,y_\ve) - \psi_0(y_0) \\
&= \ve\varphi(x_\ve,0) + O(\ve^2)
\end{align*}
so that $\Phi_\ve(x_\ve,y_\ve) >0$ if $\max \varphi(\cdot,0)>0$ and likewise $\Phi_\ve(x_\ve,y_\ve) <0$ if $\max \varphi(\cdot,0)<0$, for $\ve>0$ small enough. Assume that $\max\varphi(\cdot,y_0)>0$, and for $\delta>0$ small we define the level set $\{\Phi_\ve(x,y) = (1-\delta_0)\Phi_\ve(x_\ve,y_\ve) \}$ (if instead $\max\varphi(\cdot,0)<0$ one would consider the level set $\{\Phi_\ve(x,y) = (1+\delta_0)\Phi_\ve(x_\ve,y_\ve) \}$). Along this set, $(x,y)$ satisfies
\begin{equation}
 \delta \Phi_\ve(x_\ve,y_\ve) =  w(x,y)^2 - \frac{ D(x,y)}{H_{yy}(x,y)} (x-x_\ve)^2
 \sim ((y-y_\ve) + \ve (x-x_\ve))^2 + D(x,y) (x-x_\ve)^2,
 \label{level}
\end{equation}
while also
\begin{align*}
\Phi_\ve(x_\ve,y_\ve) &= \ve\varphi(x_0,y_0) + O(\ve^2) + O(\ve|x_\ve - x_0|)
\end{align*}
with $\varphi(x_0,y_0)>0$. Thus \eqref{level} reads
\begin{equation}
 \delta \ve\varphi(x_0,y_0) \sim ((y-y_\ve) + \ve (x-x_\ve))^2 + \ve (x-x_\ve)^2 + \ve (x-x_\ve)^2(|x-x_\ve| +|y-y_\ve|).
 \label{}
\end{equation}
Provided $\delta$ is taken small enough, this implies that $|x-x_\ve| \lesssim \delta^{1/2}$ along the streamline (if $|x-x_\ve|$ were much larger then we could not have $ (x-x_\ve)^2 + (x-x_\ve)^3 \sim \delta$),
 and so the above reads
 \begin{equation}
  \delta \ve \sim (y-y_\ve)^2 + \ve (x-x_\ve)^2,
  \label{}
 \end{equation}
which says that the streamline $\{\Phi_\ve(x,y) = (1-\delta)\Phi_\ve(x_\ve,y_\ve)\}$ is close to an ellipse $(y-y_\ve)^2 + \ve (x-x_\ve)^2 = C \delta \ve$, which has height $\sim (\delta \ve)^{1/2}$ and width $\sim \delta^{1/2}$. It is also clear that for arbitrary $\delta < 1$, the islands have height
 $\sim \ve^{1/2}$ (since if $|y-y_\ve| \sim \ve^{1/2-a}$ for $a>0$ we would have $\Phi_\ve(x,y) \sim F(\psi_0(y_0))(y-y_\ve)^2 + O(\ve) \sim F(\psi_0(y_0))\ve^{1-2a} < 0$ so we can't have $\Phi_\ve(x,y) = (1-\delta_0)\Phi_\ve(x_\ve,y_\ve) > 0$. They could however (and, in general, will) have large width.
 \end{proof}

\subsection{Genericity of elliptical islands} 
To determine that generic boundary perturbations $(g,h)\in (C^{2,\alpha}(\mathbb{T}))^2$ typically lie in $\mathcal{B}_0$, we first investigate further properties of $\varphi=\varphi_{g,h}$. In particular, we study $\varphi$ when $h$ is concentrated in a single Fourier mode and $g$ is constant.

\begin{lemma}\label{lemma:formulavarphihg}
Let $x_0\in \mathbb{T}$ and $g,h\in C^{2,\alpha}(\mathbb{T})$ such that $h(x)=\beta\cos(k(x-x_0))$, for some $\beta\in\R$, and $g(x)=g_0\in \R$. Let $\varphi$ be the unique solution to \eqref{eq:HGvarphi}. Then,
\begin{align*}
\varphi(x,y)= -\beta\psi_0'(1)\phi(y)\cos(k(x-x_0)).
\end{align*}
Here $\phi(y)$ is the unique solution to 
\begin{equation}\label{eq:phi}
\begin{cases}
\phi''(y) - (k^2 + F'(\psi_0(y))\phi(y) = 0 & y\in[-1,1], \\
\phi(-1)= 0, \quad \phi(1)=1,
\end{cases}
\end{equation}
and is such that $\phi(y_0)\neq 0$. In particular, for $-\beta\psi_0'(1)\phi(y_0)>0$, we have $\varphi(x_n,y_0)>0$ and $\partial_x^2\varphi(x_n,y_0) < 0$ for all $x_n\in\mathbb{T}$ such that $\cos(k(x_n-x_0))=1$. 
\end{lemma}

\begin{proof}
Let $f(x) = -\beta\psi_0'(1)\cos(k(x-x_0))$. Writing $\varphi(x,y) = f(x)\phi(y)$, we see that $\phi(y)$ uniquely solves \eqref{eq:phi}. If $\phi(y_0)=0$, then $\phi(y)$ solves
\begin{equation}
\begin{cases}
\phi''(y) - (k^2+F'(\psi_0(y)))\phi(y) = 0 & y\in[-1,y_0], \\
\phi(y_0)= \phi(-1)=0. & 
\end{cases}
\end{equation}
Since $F$ is elliptically stable, we see that $\phi(y) = 0$, for all $y\in [-1,y_0]$. In particular, there holds $\phi(y_0)=0$ and $\partial_y\phi(y_0)=0$. By uniqueness of solutions to  \eqref{eq:phi}, we conclude that $\phi(1)=0$ as well, thus reaching a contradiction. 
\end{proof}

\begin{remark}
Under the assumption of trivial homology, $y_0=0$ and the lemma remains true for any $(g,h)$ such that $h(x)+g(x)=2\beta\cos(k(x-x_0))$.
\end{remark}

We next show that generic boundary perturbation $(g,h)\in (C^{2,\alpha}(\mathbb{T}))^2$ belong to $\mathcal{B}_0$.

\begin{lemma}\label{lemma:nondegmaxboundary}
The set $\mathcal{B}_0$ is open and dense in the space of $(C^{2,\alpha}(\mathbb{T}))^2$ functions.
\end{lemma}

\begin{proof}
We first show that $\mathcal{B}_0$ is open. For this, let $\varphi:=\varphi_{g,h}$ and let $\widetilde{\varphi}:=\varphi_{\widetilde{g},\widetilde{h}}$  be the solutions to \eqref{eq:HGvarphi} associated to the pairs $(g,h)$ and $(\widetilde{g},\widetilde{h})$, respectively. Classical Schauder estimates give
\begin{align*}
\Vert \varphi - \widetilde{\varphi} \Vert_{C^{2,\alpha}(\mathbb{T}\times[-1,1]} \lesssim \Vert h - \widetilde{h}\Vert_{C^{2,\alpha}(\mathbb{T})} + \Vert g - \widetilde{g}\Vert_{C^{2,\alpha}(\mathbb{T})}.
\end{align*}
In particular, $\max \widetilde{\varphi}\neq 0$ for $(\widetilde{h},\widetilde{g})$ close to $(g,h)$. Likewise, for any maximum $x_0$ of $\varphi(\cdot,0)$, since $\partial_x^2\varphi(x_0,0)<0$ this shows that $\partial_x^2\widetilde{\varphi}(x,0)<0$ in a neighbourhood of $x_0$ for all $\widetilde{h}$ and $\widetilde{g}$ sufficiently close to $h$ and $g$, respectively. On the other hand, since any maximum $\widetilde{x}_0$ of $\widetilde{\varphi}(\cdot,0)$ must be arbitrarily close to maximizers of $\varphi(\cdot,0)$, for all $\widetilde{h}$ and $\widetilde{g}$ sufficiently close to $h$ and $g$, we see that  $\partial_x^2\widetilde{\varphi}(\widetilde{x}_0,0)<0$ for any such $\widetilde{x}_0$. With this, we show that $(\widetilde{g},\widetilde{h})\in \mathcal{B}_0$.

To show that $\mathcal{B}_0$ is dense, let $(g,h)\in (C^{2,\alpha}(\mathbb{T}))^2 \setminus\mathcal{B}_0$ and let $\varphi=\varphi_{g,h}$ denote the unique solution to \eqref{eq:HGvarphi}. Assume there exists some maximum $x_0\in \mathbb{T}$ of $\varphi(x,y_0)$ such that $\partial_x^2\varphi(x_0,y_0)=0$ or $\varphi(x_0,y_0)=0$. Let $\beta\in \R\setminus \lbrace 0 \rbrace$ and $k\geq 1$. For $\widetilde{h}_\beta(x)  = h(x) + \beta\cos(k(x-x_0))$ and $\widetilde{g}_\beta(x)  = g(x)$, let $\widetilde{\varphi}(x,y)=\varphi_{\widetilde{h},\widetilde{g}}(x,y)$. We have from Lemma \ref{lemma:formulavarphihg} that
\begin{align*}
\widetilde{\varphi}(x,y_0) = \varphi(x,0) -\psi_0'(1)\beta\phi(y_0)\cos(k(x-x_0)),
\end{align*}
with $\phi(y_0)\neq 0$. We further choose $\beta$ such that $-\psi_0'(1)\beta\phi(y_0) >0$. Then, $\widetilde{\varphi}(\cdot,y_0)$ has a global positive non-degenerate maximum at $x_0$. Note that any other global maxima of $\widetilde{\varphi}$ must be also global maximums of $\cos(k(x-x_0))$, hence they are non-degenerate and
\begin{align*}
\varphi(x_*,y_0) < \varphi(x_*,y_0) - \beta\psi_0'(1)\phi(y_0)= \widetilde{\varphi}(x_*,y_0) 
\end{align*}
so that $\widetilde{\varphi}(x_*,y_0) >0$ if $\varphi(x_*,y_0)\geq 0$ and $\widetilde{\varphi}(x_*,y_0) <0$ if $\varphi(x_*,y_0)< 0$, for $\beta>0$ small enough, for all global maximum $x_*$ of $\widetilde{\varphi}(x,y_0)$.  We conclude that $(\widetilde{h}_\beta,\widetilde{g}_\beta)\in \mathcal{B}_0$, for all $\beta$ small enough. 
\end{proof}

\appendix
\section{Islands of fluid in special cases} \label{appendislands1}

In this section, we discuss some cases for which one can ensure the existence of islands, irrespective of the size of the boundary perturbation. These results complement the generic and perturbative results of the main text and while we state them for periodic domains defined by graphs, an inspection of the proofs show that they are likewise true for general annular domains. For $g,h\in C^{2,\alpha}(\mathbb{T})$ with $g(x) < h(x)$, set
\begin{equation}
D_{g,h} := \lbrace (x,y): x\in \mathbb{T}, \, g(x) \leq y \leq h(x) \rbrace,
\end{equation}
and for any $F\in C^1$ and any $c_h,c_g\in \R$, consider the solution $\psi$ (if any), to
\begin{equation}\label{eq:steadyEulerD}
\begin{cases}
\Delta\psi = F(\psi) & \text{ in } D_{G,H}, \\
\psi(x, h(x))=c_h & \text{ for } x\in\mathbb{T}, \\
\psi(x,g(x))=c_g & \text{ for } x\in\mathbb{T}.
\end{cases}
\end{equation}
The first result states that, in the presence of stagnation points, fixing one boundary to be horizontal and perturbing the other always produces islands. 

\begin{proposition}\label{prop:ctnbottombdry}
Assume $g'\equiv 0$ and $h'\neq 0$. Let $\psi$ be any solution to \eqref{eq:steadyEulerD} which stagnates. Then, $\psi$ possess islands.
\end{proposition}

\begin{proof}
Assume otherwise, namely $\psi$ is laminar and non-constant.  By criticality, $\nabla\psi=0$ on some closed and non-contractible $\Gamma$. Then, we may consider $u=\nabla^\perp\psi$ as a steady laminar Euler solution on $D_{g,\Gamma}$, the region in $D_{g,h}$ bounded above by $\Gamma$. Given the overdetermined nature of the boundary conditions on $\Gamma$ and the fact that the bottom boundary of $D_{g,\Gamma}$ is flat, Corollary 1.1 in \cite{drivas2024geometric} shows that $\Gamma$ is a flat horizontal curve as well and $\nabla\psi=0$ there. Next, since $F$ is Lipschitz, one can use Lemma 3 in \cite{drivas2024islands} and the fact that $h'\neq 0$ to obtain that $\psi$ is constant, a contradiction.
\end{proof}

When both $h$ and $g$ are non-trivial perturbations, the appearance of islands is ensured if solutions attain more than one critical value.

\begin{proposition}\label{prop:islandscritical}
Let $g'\neq 0$ and $h'\neq 0$. Suppose $\psi$ is a solution to \eqref{eq:steadyEulerD} whose set of critical points is disconnected. Then, $\psi$ has islands.
\end{proposition}

\begin{proof}
If $\psi$ is laminar, there exists $c_1,c_2$ and closed non-contractible and disjoint level sets $\Gamma_{c_i}$ such that  $\psi|_{\Gamma_{c_i}}=c_i$ and  $(\nabla\psi)|_{\Gamma_{c_i}}=0$. Then, $\psi$ is a steady laminar solution of Euler in $D_{\Gamma_{c_1},\Gamma_{c_2}}$, say, with overdetermined boundary conditions. Hence, Corollary 1.1 in \cite{drivas2024geometric} shows that $\Gamma_{c_i}$ are flat horizontal curves where $\nabla\psi = 0$. Hence, \cite[Lemma 3]{drivas2024islands} shows that, since $h'\neq 0$ and $g'\neq 0$, then $\psi$ must be constant in $D_{g,h}$, a contradiction.
\end{proof}

A first consequence of the proposition is for non-linearities $F$ for which the associated solution $\psi_0$ to \eqref{eq:steadyEulerD} in $\mathbb{T}\times[-1,1]$ has more than one critical point. Then, the solution associated to any non-trivial boundary perturbation will necessarily develop islands. 

A second corollary of the proposition is that all eigenfunctions but the first of the Dirichlet Laplacian in $D_{g,h}$ with at least one non-trivial boundary must have have islands. Indeed, since the first eigenfunction of the Dirichlet Laplacian does not change sign and the set of eigenfunctions is $L^2$ orthogonal, all other eigenfunctions do change sign. In particular, they attain their maximum and minimum in $D_{g,h}$ and the proposition applies. This situation is to be compared with that of $\mathbb{T}\times[-1,1]$, where there exists shear-flow eigenfunctions $\phi_n=\phi_n(y)$ associated to infinitely many eigenvalues.

Finally, we remark the following result on partially ``free-boundary"  equilibria (see \cite{CDG22}):
 \begin{lemma}\label{lemma:islandsingularbdry}
If $g'\neq 0$ and $|\nabla\psi| = a\geq 0$ on $\partial D_g$ then, any stagnant solution $\psi$ possess islands.
\end{lemma}
\begin{proof}
Assume again towards a contradiction that $\psi$ is laminar. Then, there existence of some curve $\Gamma$ where $\nabla\psi=0$. Next, since $\psi$ is a laminar steady solution of Euler in $D_g^\Gamma$ and enjoys overdetermined boundary conditions for $\psi$, recall that $\nabla\psi=0$ on $\Gamma$ and $|\nabla\psi|=a\geq 0$ on $\partial D_g$,  Corollary 1.1 in \cite{drivas2024geometric} implies both $\partial D_g$ and $\Gamma$ are flat horizontal curves, contradicting $g'\neq 0$.
\end{proof}

\section{Estimates on the error function}\label{estappend}
This section is devoted to the proof of Proposition \ref{prop:boundsrve} which states that
\begin{equation}\label{eq:rveappendix}
\begin{cases}
\left(\Delta-F'(\psi_0)\right)r_\ve = \widetilde{\N}(r_\ve)(x,y) & \text{ in } D_\ve, \\
r_\ve(x,H(x)+\ve h(x))= r_\ve^{top}(x) & \text{ for } x\in \mathbb{T}, \\
r_\ve(x, G(x)+\ve g(x))= r_\ve^{bot}(x) & \text{ for } x\in\mathbb{T}, \\
\end{cases}
\end{equation}
where
\begin{align*}
\widetilde{\N}(v) = F(\psi_0 + \ve\varphi + v) - F'(\psi_0)r_\ve -\Delta\psi_0 - \ve\Delta \varphi
\end{align*}
and
\begin{align*}
r_\ve^{top}(x)&:= c_H-  \psi_0(x,H(x)+\ve h(x)) - \ve  \varphi(x,H(x)+\ve h(x)), \\
r_\ve^{bot}(x)&:= c_G-  \psi_0(x,G(x)+\ve g(x))- \ve\varphi(x,G(x)+\ve g(x)) 
\end{align*}
admits a unique solution $r_\ve$ with $\Vert r_\ve \Vert_{C^{2,\alpha}(D_\ve)}\lesssim \ve^2$, for $\ve>0$ small enough. Define
\begin{align*}\label{eq:defeta}
\eta_\ve(x,y) := r_\ve^{top}(x)\frac{y-G(x)-\ve g(x)}{H(x)-G(x) + \ve h(x) - \ve g(x)} + r_\ve^{bot}(x)\frac{H(x)+ \ve h(x) - y}{H(x)-G(x) + \ve h(x) - \ve g(x)}, 
\end{align*}
which linearly interpolates between the two boundary conditions and further denote $\beta_\ve(x,y) := (\Delta - F'(\psi_0))\eta_\ve(x,y)$. Clearly,
\begin{align*}
\Vert \eta_\ve \Vert_{C^{2,\alpha}(D_\ve)}\lesssim \ve^2, \quad \Vert \beta_\ve\Vert_{C^{0,\alpha}(D_\ve)}\lesssim \ve^2.
\end{align*}
Then, for $u_\ve(x,y) = r_\ve(x,y) - \eta_\ve(x,y)$, we have that
\begin{equation}\label{eq:uve}
\begin{cases}
\left(\Delta-F'(\psi_0)\right)u_\ve = \N(u_\ve) - \beta_\ve   & \text{ in } D_\ve, \\
u_\ve = 0 & \text{ on } \partial D_\ve, 
\end{cases}
\end{equation}
with
\begin{equation}\label{eq:defNappendix}
\begin{split}
\N(v)(x,y) &= F(\psi_0 + \ve{\varphi} + v + \eta_\ve) - F(\psi_0) - F'(\psi_0)\left(\ve {\varphi} + v + \eta_\ve\right) \\
&\quad +(F(\psi_0) - \psi_0'') -\ve (\Delta - F'(\psi_0)){\varphi}.
\end{split}
\end{equation}
Observe that both $F(\psi_0)-\psi_0''$ and $(\Delta - F'(\psi_0))\varphi$ identically vanish on $D_0$, while for $(x,y)\in D_\ve\setminus D_0$, they are $\ve$ small, since we can use Taylor expansion nearby $\partial D_0$. Proposition \ref{prop:boundsrve} is a consequence of the following result for $u_\ve$.

\begin{proposition}\label{prop:boundsuve}
Let $\alpha\in (0,1)$. Let $g,h\in C^{2,\alpha}(\mathbb{T})$ and $F\in C^4$. Then, there $\ve_0>0$ such that for all $0<\ve < \ve_0$, there exists a unique solution $u_\ve$ to \eqref{eq:uve} and it is such that
\begin{align*}
\Vert u_\ve \Vert_{C^{2,\alpha}(D_\ve)} \lesssim \ve^2 \left( \Vert F \Vert_{C^4} + \Vert h \Vert_{C^{2,\alpha}} + \Vert g \Vert_{C^{2,\alpha}} \right).
\end{align*}
In particular, there holds
\begin{align*}
\Vert r_\ve \Vert_{C^{2,\alpha}(D_\ve)} \lesssim \ve^2 \left( \Vert F \Vert_{C^4} + \Vert h \Vert_{C^{2,\alpha}} + \Vert g \Vert_{C^{2,\alpha}} \right)
\end{align*}
for all $0 < \ve < \ve_*$.
\end{proposition} 

The proof is straightforward and follows from a fixed point argument. For this, we define 
\begin{align*}
X_\ve = \left \lbrace v\in C^{0,\alpha}(D_\ve) : \Vert v \Vert_{C^{0,\alpha}(D_\ve)} \leq  \ve \right\rbrace
\end{align*}
together with the operators 
\begin{align*}\label{def:K}
K_\ve(v)(x,y) := \N(v)(x,y) - \beta_{\ve}(x,y), \quad \widetilde{K}_\ve(v) := (\Delta- F'(\psi_0))_{\mathrm{D}}^{-1}K_\ve(v),
\end{align*}
where here $w:=(\Delta- F'(\psi_0))_{\mathrm{D}}^{-1}f$ denotes the solution operator of the elliptic problem
\begin{equation}\label{eq:w}
\begin{cases}
\left(\Delta-F'(\psi_0)\right)w = f   & \text{ in } D_\ve, \\
w = 0 & \text{ on } \partial D_\ve.
\end{cases}
\end{equation}

To apply the Banach fixed point argument, we first show that $\widetilde{K}_\ve$ maps $X_\ve$ to itself.

\begin{lemma}\label{lemma:invariantX}
There exists $\ve_*>0$ such that $\widetilde{K}_\ve(v)\in X_\ve$, for all $v\in X_\ve$, for all $0 < \ve < \ve_*$.
\end{lemma}

\begin{proof}
Let $\ve>0$ and $v\in X_\ve$, with that $\Vert v \Vert_{C^{0,\alpha}(D_\ve)}\leq \ve$. By Schauder standard estimates
(see e.g. \cite[Theorem 6.6]{gilbargtrudinger}), we have that
\begin{align*}
\Vert \widetilde{K}_\ve(v) \Vert_{C^{2,\alpha}} \leq c_0 \Vert  N(v) -\beta_\ve \Vert_{C^{0,\alpha}}.
\end{align*}
Here the Schauder constant $c_0=c_0(D_\ve)$ depends in principle on $D_\ve$ and thus also on $\ve$. However, since $D_\ve$ is a small and smooth perturbation of $D_0$, one can show that $c_0$ is uniformly bounded in $\ve$; we omit the details. Now, since $\Vert \beta_\ve \Vert_{C^{0,\alpha}(D_\ve)}\lesssim \ve^2$, there holds $c_0\Vert \beta_\ve \Vert_{C^{0,\alpha}(D_\ve)} \leq \frac{\ve}{2}$, for $\ve$ small enough. Next we shall see that $c_0\Vert \N(u)\Vert_{C^{0,\alpha}(D_\ve)}\leq \frac{\ve}{2}$. To do so, we note that 
\begin{equation}\label{eq:usefulN}
\N(u) = (\ve\varphi + \eta_\ve + u)^2\int_0^1 \int_0^t F''(\psi_0 + s(\ve\varphi + \eta_\ve + u)) \rmd s \rmd t.
\end{equation}
In particular,
\begin{align*}
\Vert \N(u)\Vert_{L^\infty(D_\ve)}\leq \Vert F'' \Vert_{L^\infty}\left( \Vert \ve\varphi + \eta_\ve \Vert_{L^\infty(D_\ve)}^2 + \Vert v \Vert_{L^\infty(D_\ve)}^2 \right) \leq \frac{\ve}{4c_0}
\end{align*}
for $\ve>0$ small enough. Likewise, we have
\begin{align*}
\left[ \N(v) \right]_{C^\alpha} &\leq 2\Vert F'' \Vert_{L^\infty}\left( \Vert \ve\varphi + \eta_\ve \Vert_{C^{0,\alpha}(D_\ve)}^2 + \Vert v \Vert_{C^{0,\alpha}(D_\ve)}^2 \right) \\
&\quad+ \Vert F ''' \Vert_{L^\infty}\left( \Vert \ve\varphi + \eta_\ve \Vert_{L^\infty(D_\ve)}^2 + \Vert v \Vert_{L^\infty(D_\ve)}^2 \right)\leq \frac{\ve}{4c_0},
\end{align*}
again for $\ve>0$ small enough. With this, we conclude that $\Vert \widetilde{K}_\ve(v) \Vert_{C^{0,\alpha}(D_\ve)}\leq \frac{\ve}{2}$, for all $v\in X_\ve$, for $\ve>0$ small enough.
\end{proof}

We next show that $\widetilde{K}_\ve$ is a contraction in $X_\ve$,  for $\ve>0$ small enough.
\begin{lemma}\label{lemma:contraction}
There exists $\ve_*>0$ such that 
\begin{align*}
\Vert \widetilde{K}_\ve(u) - \widetilde{K}_\ve(v) \Vert_{C^{0,\alpha}(D_\ve)} < \frac12 \Vert u - v \Vert_{C^{0,\alpha}(D_\ve)},
\end{align*}
for all $u,v\in X_\ve$ and all $0< \ve < \ve_*$.
\end{lemma}

\begin{proof}
For  $u,v\in X_\ve$ let $w:=\widetilde{K}_\ve(u) - \widetilde{K}_\ve(v)$, we have that 
\begin{equation}
\begin{cases}
\Delta w  - F'(\psi_0)w=  \N(u) - \N(v)  & \text{ in } D_\ve, \\
w = 0 & \text{ on } \partial D_\ve.
\end{cases}
\end{equation}
Appealing to the Schauder estimates of Theorem 6.6 in \cite{gilbargtrudinger}, we have that
\begin{align*}
\Vert w \Vert_{C^{2,\alpha}} \leq c_0 \Vert N(u) - N(v) \Vert_{C^{0,\alpha}}.
\end{align*}
Next, from \eqref{eq:usefulN}, there holds
\begin{align*}
\N(u) - \N(v) &= \left( (\ve\varphi + \eta_\ve + u)^2 - (\ve\varphi + \eta_\ve + v)^2 \right) \int_0^1 \int_0^t F''(\psi_0 + s(\ve\varphi + \eta_\ve + u)) \rmd s \rmd t \\
&\quad +(\ve\varphi + \eta_\ve + v)^2 \left( \int_0^1 \int_0^t F''(\psi_0 + s(\ve\varphi + \eta_\ve + u)) - F''(\psi_0 + s(\ve\varphi + \eta_\ve + v)) \rmd s \rmd t \right). \\
&= (u + v + 2\varphi + 2\eta_\ve)(u-v) \int_0^1 \int_0^t F''(\psi_0 + s(\ve\varphi + \eta_\ve + u)) \rmd s \rmd t \\
&\quad + (\ve\varphi + \eta_\ve + v)^2 (u-v) \int_0^1 \int_0^t \int_0^1 sF'''(\psi_0 + s(\ve\varphi + \eta_\ve + v + r(u-v))) \rmd r \rmd s \rmd t \\
&=\mathcal{H}_\ve(u,v) (u-v).
\end{align*}
In particular, we see that
\begin{align*}
\Vert \N(u) - \N(v) \Vert_{L^\infty(D_\ve)}  &\leq \frac12\Vert F'' \Vert_{L^\infty}\left( 2\ve \Vert \varphi \Vert_{L^\infty(D_\ve)} + 2\Vert \eta \Vert_{L^\infty(D_\ve)} +  \Vert u \Vert_{L^\infty(D_\ve)} + \Vert v \Vert_{L^\infty(D_\ve)} \right) \Vert u - v \Vert_{L^\infty(D_\ve)} \\ 
&\quad +\frac12 \Vert F'''\Vert_{L^\infty}\left( \ve^2 \Vert \varphi \Vert_{L^\infty(D_\ve)}^2 +  \Vert \eta \Vert_{L^\infty(D_\ve)}^2  + \Vert v \Vert_{L^\infty(D_\ve)}^2 \right) \Vert u - v \Vert_{L^\infty(D_\ve)},
\end{align*}
so that, for $\ve$ sufficiently small, we deduce
\begin{align*}
\Vert \N(u) - \N(v) \Vert_{L^\infty(D_\ve)}  &\leq \frac{1}{4c_0}\Vert u - v \Vert_{L^\infty(D_\ve)}.
\end{align*}
On the other hand, for the $C^{0,\alpha}$ semi-norm, we have
\begin{align*}
\left[ \N(u) - \N(v) \right]_{C^\alpha} \leq \Vert\mathcal{H}_\ve(u,v)\Vert_{L^\infty}\left[ u - v \right]_{C^\alpha} + \Vert u-v \Vert_{L^\infty} \left[ \mathcal{H}_\ve(u,v) \right]_{C^\alpha}
\end{align*}
where we recall $\Vert \mathcal{H}_\ve (u,v) \Vert_{L^\infty(D_\ve)}  \leq \frac{1}{4c_0}$ while now, since $\Vert \ve\varphi + \eta_\ve \Vert_{C^{0,\alpha}(D_\ve)}<1$, and also $\Vert u \Vert_{C^{0,\alpha}(D_\ve)}<1$ and $\Vert v \Vert_{C^{0,\alpha}(D_\ve)} < 1$, there holds
\begin{align*}
\left[ \mathcal{H}_\ve(u,v) \right]_{C^\alpha} &\leq C \left( \Vert \ve \varphi + \eta_\ve \Vert_{C^{0,\alpha}(D_\ve)} +  \Vert u \Vert_{C^{0,\alpha}(D_\ve)} + \Vert v \Vert_{C^{0,\alpha}(D_\ve)} \right) 
\end{align*}
for some $C>0$, for all $\ve\in (0,1)$. In particular, for $\ve>0$ small enough, we have that 
\begin{align*}
\left[ \mathcal{H}_\ve(u,v) \right]_{C^\alpha} &\leq \frac{1}{4c_0},
\end{align*} 
for all $u\in X_\ve$. With this, we conclude that $\left[ \N(u) - \N(v) \right]_{C^\alpha} \leq \frac{1}{2c_0}\Vert u - v \Vert_{C^{0,\alpha}(D_\ve)}$, and together with the $L^\infty$ estimate we obtain 
\begin{align*}
 \Vert \widetilde{K}_\ve(u) - \widetilde{K}_\ve(v) \Vert_{C^{2,\alpha}(D_\ve)} \leq c_0 \Vert \N(u) - \N(v) \Vert_{C^{0,\alpha(D_\ve)}} \leq \frac12\Vert u - v \Vert_{C^{0,\alpha}(D_\ve)},
\end{align*}
for $\ve>0$ small enough. Further restricting to a smaller $\ve$ if necessary, this shows that $\widetilde{K}_\ve$ is a contraction in $X_\ve$.
\end{proof}

\begin{proof}[Proof of Proposition \ref{prop:boundsuve}]
From Lemma \ref{lemma:invariantX} and Lemma \ref{lemma:contraction} there exists $\ve_*>0$ such that $\widetilde{K}_\ve:X_\ve \rightarrow X_\ve$ and it is a contraction for all $0 < \ve <\ve_*$. As a result, for all $\ve\in (0,\ve_*)$, by the Banach fixed point theorem, there exists a unique $u_\ve\in X_\ve$ with $\Vert u_\ve \Vert_{C^{0,\alpha}(D_\ve)}\leq \ve$ and such that $\widetilde{K}_\ve(u_\ve) = u_\ve$. In particular, from \eqref{eq:usefulN}, we observe that
\begin{align*}
\Vert u_\ve \Vert_{C^{2,\alpha}} \leq c_0 \left( \Vert \N(u_\ve) \Vert_{C^{0,\alpha}(D_\ve)} + \Vert \beta_\ve \Vert_{C^{0,\alpha}(D_\ve)}\right) \lesssim \ve^2 \left( \Vert F \Vert_{C^3} + \Vert h \Vert_{C^{2,\alpha}} + \Vert g \Vert_{C^{2,\alpha}} \right)
\end{align*}
and the proof is concluded.
\end{proof}

 \subsection*{Acknowledgments}   We thank Tarek Elgindi and Yupei Huang for useful discussions.
The research of TDD was partially supported by the NSF DMS-2106233 grant, NSF CAREER award \#2235395, a Stony Brook University Trustee's award as well as an Alfred P. Sloan Fellowship. The research of
DG was partially supported by a startup grant from Brooklyn College, PSC-CUNY grant TRADB-55-214, and
the NSF grant DMS-2406852.
The research of MN was partially supported by the European Research Council (ERC) under the European Union’s Horizon 2020 research and innovation programme through the grant agreement 862342.

\bibliographystyle{abbrv}
\bibliography{refs}

\end{document}